\documentclass[english, oneside, 12pt]{smfart}

\usepackage{smfthm}
\usepackage{amsmath}
\usepackage{amsfonts}
\usepackage{latexsym}
\usepackage{mathrsfs}
\usepackage{fancybox}
\usepackage{times} 
\usepackage{txfonts} 
\usepackage{array} 
\usepackage{hyperref}
\hypersetup{
    colorlinks=true,
    linkcolor=blue,
    filecolor=magenta,      
    urlcolor=blue,
citecolor = black,
}
\usepackage[mathcal]{euscript}
\usepackage{amssymb}
\usepackage{epsf, epic, eepic} 
\usepackage{graphicx}
\usepackage{wrapfig}
\usepackage{tikz}

\theoremstyle{remark}
\newtheorem{exer}{Exercise}
\input xy 
\xyoption{all} 
\newcommand{\fullwebaddressold}[1]{http://arxiv.org/abs/math/#1}
\newcommand{\siderefboxold}[1]
{\mbox{}\marginpar{\hspace{0pt}\small{{   \href{\fullwebaddressold#1}{{#1}}}}}}
\newcommand{\newciteold}[2]{\cite{{#1}}{#2} \hspace{-1.1mm}\siderefboxold{{#1}}}

\newcommand{\fullwebaddress}[1]{http://arxiv.org/abs/#1}
\newcommand{\siderefbox}[1]
{\mbox{}\marginpar{\hspace{0pt}\small{{   \href{\fullwebaddress#1}{{#1}}}}}}
\newcommand{\newcite}[2]{\cite{{#1}}{#2} \hspace{-1.1mm}\siderefbox{{#1}}}

\DeclareMathOperator{\im}{im}

\newcommand{\invlim}{\varprojlim}

\newcommand{\ra}{\rightarrow}

\newcommand{\leftsq}{[\![}
\newcommand{\rightsq}{]\!]}

\newcommand{\bbF}{\mathbb{F}}
\newcommand{\bbZ}{\mathbb{Z}}
\newcommand{\bbB}{\mathbb{B}}
\newcommand{\bbQ}{\mathbb{Q}}
\newcommand{\bbC}{\mathbb{C}}
\newcommand{\bbR}{\mathbb{R}}
\newcommand{\bbK}{\mathbb{K}}
\newcommand{\bbP}{\mathbb{P}}

\newcommand{\cC}{\mathcal{C}}

\newcommand{\links}{{\mathbf{Links}}}

\newcommand{\moy}{\mathbf{MOY}}

\newcommand{\modR}{\mathbf{Mod}_R}
\newcommand{\vectR}{\mathbf{Vect}_R}
\newcommand{\grmodR}{\mathbf{GrMod}_R}
\newcommand{\grvectQ}{\mathbf{GrVect_{\bbQ}}}

\newcommand{\prekh}{Kh}
\newcommand{\prerkh}{\widetilde{Kh}}
\newcommand{\prekhodd}{Kh_{\text{odd}}}
\newcommand{\prekhoddr}{\widetilde{Kh}_{\text{odd}}}
\newcommand{\kh}[3]{Kh^{{#1},{#2}}({#3})}
\newcommand{\rkh}[3]{\widetilde{Kh}^{{#1},{#2}}({#3})}
\newcommand{\khodd}[2]{Kh_{\text{odd}}^{{#1},{#2}}}
\newcommand{\khoddr}[2]{\widetilde{Kh}_{\text{odd}}^{{#1},{#2}}}

\newcommand{\preikh}{Kh_\bbZ}
\newcommand{\preirkh}{\widetilde{Kh}_\bbZ}
\newcommand{\ikh}[3]{Kh_\bbZ^{{#1},{#2}}({#3})}
\newcommand{\irkh}[3]{\widetilde{Kh}_\bbZ^{{#1},{#2}}({#3})}

\newcommand{\khkh}[4]{Kh_{#3}^{{#1},{#2}}({#4})}
\newcommand{\rkhkh}[4]{\widetilde{Kh}_{#3}^{{#1},{#2}}({#4})}

\newcommand{\lee}[1]{Lee^{#1}}
\newcommand{\kr}[2]{KR_{N}^{{#1},{#2}}}
\newcommand{\krr}[2]{\widetilde{KR}_{N}^{{#1},{#2}}}

\newcommand{\barh}{\widetilde{H}}

\newcommand{\tang}{{\mathbf{Tang}}}

\newcommand{\vectftwo}{{\mathbf{Vect}}_{\bbF_2}}
\newcommand{\vectq}{{\mathbf{Vect}}_{\bbQ}}

\newcommand{\posx}{\hspace{1mm}\raisebox{-1.3mm}{\includegraphics[scale=0.07]{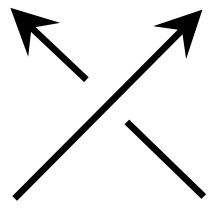}}\hspace{1mm}}
\newcommand{\negx}{\hspace{1mm}\raisebox{-1.3mm}{\includegraphics[scale=0.07]{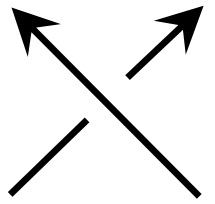}}\hspace{1mm}}
\newcommand{\ores}{\hspace{1mm}\raisebox{-1.3mm}{\includegraphics[scale=0.07]{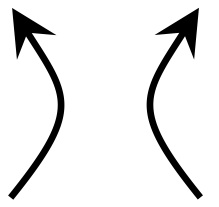}}\hspace{1mm}}
\newcommand{\sres}{\hspace{1mm}\raisebox{-1.3mm}{\includegraphics[scale=0.07]{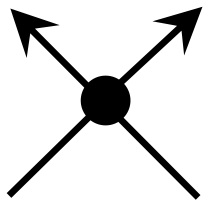}}\hspace{1mm}}
\newcommand{\uoreszero}{\hspace{1mm}\raisebox{-1.3mm}{\includegraphics[scale=0.07]{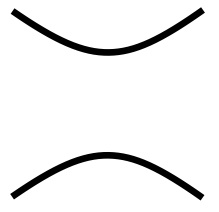}}\hspace{1mm}}
\newcommand{\uoresone}{\hspace{1mm}\raisebox{-1.3mm}{\includegraphics[scale=0.07]{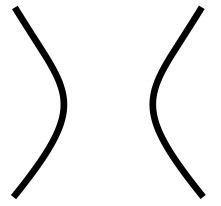}}\hspace{1mm}}
\newcommand{\uox}{\hspace{1mm}\raisebox{-1.3mm}{\includegraphics[scale=0.045]{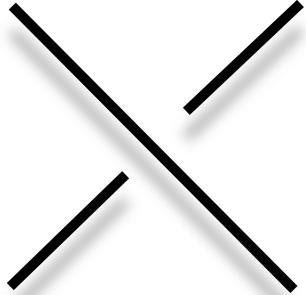}}\hspace{1mm}}

\title{A hitchhiker's guide to Khovanov homology}

\dedicatory{I dedicate these notes to the memory of Ruty Ben-Zion}

\author{Paul Turner}
\address{ Section de math\'ematiques,  
Universit\'e de Gen\`eve, 2-4 rue du Li\`evre, CH-1211, Geneva, Switzerland. }
\email{prt.maths@gmail.com}

\begin{document}
\cornersize{0.05}
\maketitle

\begin{abstract}
These notes from the 2014 summer school Quantum Topology at the CIRM in Luminy attempt to provide a rough guide to a selection of developments in Khovanov homology over the last fifteen years.  
\end{abstract}

\maketitle

\tableofcontents

\section*{Foreword}
There are already too many introductory articles on Khovanov homology
and another is not really needed. On the other hand by now -  15 years
after the invention of subject - it is quite easy to get lost after
having taken those first few steps. What could be useful is a rough
guide to some of the developments over that time and the summer school
{\em Quantum Topology} at the CIRM in Luminy has provided the ideal
opportunity for thinking about what such a guide should look like. It is quite a risky undertaking because it is all too easy to offend
 by omission, misrepresentation or other. I have not attempted a
 complete literature survey and inevitably these notes reflects my
 personal view, jaundiced as it may often be. My apologies in advance for any
 offense caused.

In this preprint version I have
added arXiv
reference numbers as marginal annotations. If you are reading this
electronically these are live links taking you directly to the
appropriate arXiv page of the article being referred to.

I would like to express my warm thanks to David Cimasoni, Lukas Lewark, Alex Shumakovitch, Liam Watson and Ben Webster.

\section{A beginning}
There are a number of introductions to Khovanov homology. A good place
to start is Dror Bar-Natan's exposition of Khovanov's work
\begin{itemize}
\item  {\em On Khovanov's categorification of the Jones polynomial}
  (Bar-Natan, \newciteold{0201043}) ,
\end{itemize}
followed by Alex Shumakovitch's introduction 
\begin{itemize}
\item {\em Khovanov homology theories and their applications} ( Shumakovitch, \newcite{1101.5614}),
\end{itemize}
not forgetting the original paper by Mikhail Khovanov
\begin{itemize}
\item  {\em A categorification of the Jones polynomial} (Khovanov, \newciteold{9908171}).
\end{itemize}
Another possible starting point is 
\begin{itemize}
\item {\em Five lectures on Khovanov homology} (Turner, \newciteold{0606464}).
\end{itemize}

\subsection{There is a link homology theory called Khovanov homology}

What are the minimal requirements of something deserving of the name {\em link homology theory}? We should expect a functor
$$
H\colon \links\ra \mathbf{A}
$$
where $\links$ is some category of links in which isotopies are morphisms and $\mathbf{A}$ another category, probably abelian, where we have in mind the category of finite dimensional vector spaces, $\vectR$, or of modules, $\modR$, over a fixed ring $R$. This functor should satisfy a number of properties.
\begin{itemize}
\item {\em Invariance.} If $L_1\ra L_2$ is an isotopy then the induced map $H(L_1) \ra H(L_2)$ should be an isomorphism.
\item {\em Disjoint unions.} Given two disjoint links $L_1$ and $L_2$ we want the union expressed in terms of the parts
$$H(L_1 \sqcup L_2) \cong H(L_1) \square H(L_2)$$
where $\square$ is some monoidal operation in $\mathbf{A}$ such as $\oplus$ or $\otimes$.
\item {\em Normalisation.} The value of $H(\text{unknot})$ should be specified. (Possibly also the value of the empty knot).
\item {\em Computational tool.} We want something like a long exact sequence which relates homology of a given link with associated ``simpler'' ones - something like the Meyer-Vietoris sequence in ordinary homology. 
\end{itemize}

If these are our expectations then Khovanov homology is bound to please. Let us take $\links$ to be the category whose objects are oriented links in $S^3$ and whose morphisms are  link cobordisms, that is to say compact oriented surfaces-with-boundary in $S^3\times I$ defined up to isotopy. All manifolds are assumed to be smooth.

\begin{theo}[Existence of Khovanov homology] \label{thm:existence}  There exists a (covariant) functor
$$
\prekh \colon \links \ra \vectftwo
$$
satisfying
\begin{enumerate}
\item If $\Sigma\colon L_1\ra L_2$ is an isotopy then $\prekh(\Sigma) \colon \prekh(L_1) \ra \prekh(L_2)$ is an isomorphism,
\item $\prekh(L_1\sqcup L_2) \cong \prekh(L_1) \otimes \prekh(L_2)$,
\item $\prekh(\text{unknot}) = \bbF_2 \oplus \bbF_2$ and $\prekh(\emptyset) = \bbF_2$,
\item If $L$ is presented by a link diagram a small piece of which is $ \uox$ then there is an exact triangle
$$
\xymatrix{
\prekh(\uoresone) \ar[rr] & & \prekh(\uox) \ar[dl]\\
& \prekh(\uoreszero) \ar[ul]
}
$$
\end{enumerate}
\end{theo}

In fact a little more is needed to guarantee something non-trivial and
in addition to the above we demand that $\prekh$ carries a bigrading
$$
\kh ** L = \bigoplus_{i,j\in \bbZ} \kh ijL
$$
and with respect to this
\begin{itemize}
\item  a link cobordism $\Sigma\colon L_1\ra L_2$ induces a map $\prekh(\Sigma)$ of bidegree $(0, \chi(\Sigma))$,
\item the generators of the unknot have bidegree $(0,1)$ and $(0,-1)$ (and for the empty knot bidegree $(0,0)$),
\item the exact triangle unravels as follows:\\
\noindent
{\em Case I: $\negx$} For each $j$ there is a long exact sequence
$$
\xymatrix@=16pt{
\ar[r]^(.2){\delta} & \kh i{j+1}  \ores \ar[r] & \kh i{j}  \negx \ar[r] & \kh {i-\omega}{j-1 -3\omega}  \uoreszero \ar[r]^(0.55)\delta & \kh {i+1}{j+1}  \ores \ar[r] & 
}
$$
where $\omega$ is the number of negative crossings in the chosen orientation of $\uoreszero$  minus the number of negative crossings in $\negx$.\\
{\em Case II: $\posx$} For each $j$ there is a long exact sequence
$$
\xymatrix@=16pt{
\ar[r] & \kh {i-1}{j-1}  \ores \ar[r]^(.45){\delta} & \kh {i-1-c}{j-2-3c}  \uoreszero \ar[r] & \kh ij \posx \ar[r] & \kh i{j-1}  \ores \ar[r]^(0.8)\delta & 
}
$$
where $c$ is the number of negative crossings in the chosen orientation of $\uoreszero$ minus the number of negative crossings in $\posx$.\\
\end{itemize}

To prove the theorem one must construct such a functor, but first
let's see a few consequences relying only on existence and
standard results.

\begin{prop}
If a link $L$ has an odd number of components then $\kh *
{\text even}{L}$ is trivial. If it has an even number of components then  $\kh * {\text odd}{L}$ is trivial.
\end{prop}
\begin{proof}
  The proof is by induction on the number of crossing and uses the following elementary result.
\begin{center}
\begin{minipage}{0.9\linewidth}
  \begin{lemm}
    In the discussion of the long exact sequences above (i) if the strands featured at the crossing are from the same component then $\omega$ is odd and $c$ is even, and (ii) if they are from different components then $\omega$ is even and $c$ odd.
  \end{lemm}
 \end{minipage}
\end{center}
For the inductive step we use this and, depending on the case, one of the long exact sequence shown above, observing that in each case two of the three groups shown are trivial.
\end{proof}

\begin{prop}
 If $L^!$ denotes the mirror image of the link $L$ then $\kh i j {L^!} \cong \kh {-i}{-j} L$.
\end{prop}
\begin{proof}
  There is a link cobordism $\Sigma\colon L^! \sqcup L \ra \emptyset$
  with $\chi(\Sigma) = 0$ obtained by bending the identity cobordism  (a cylinder) $L\ra L$.
Since $\prekh$ is a functor there is an induced map of bidegree $(0,\chi(\Sigma)) = (0,0)$
$$
\Sigma_*\colon \kh ** {L^! } \otimes \kh ** {L}\ra \kh ** \emptyset = \bbF_2.
$$
By a standard ``cylinder straightening isotopy'' argument
\begin{center}
\includegraphics[scale=0.2]{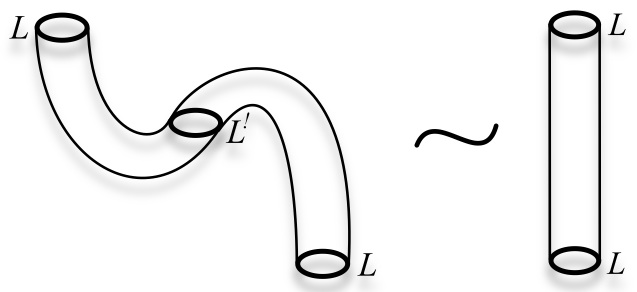}
\end{center}
 the bilinear form is non-degenerate, 
and the result follows recalling that we are in a bigraded setting so

$$
(\kh ** {L^! } \otimes \kh ** {L} )^{0,0} = \bigoplus_{i,j} \kh i j  {L^! } \otimes \kh {-i}{-j} {L}.
$$ 
\end{proof}

\begin{exer}
  Theorem \ref{thm:existence} includes the statement that
  $\prekh(\text{unknot}) = \bbF_2 \oplus \bbF_2$. In fact we could assume
  the weaker statement: the homology of the unknot is concentrated in
  degree zero. Use this along with the diagram
  \raisebox{-0.5mm}{\includegraphics[scale=0.07]{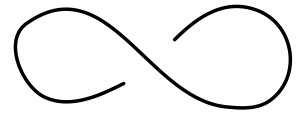}}, the long
  exact sequence, the property on disjoint unions and the invariance of Khovanov homology to show that
  dim$(\prekh(\text{unknot}))=2$.
\end{exer}

\begin{prop} For any oriented link $L$, 
 $$
\frac{1}{t^\frac12 + t^{-\frac12}}\sum_{i,j} (-1)^{i+j+1}t^{\frac{j}2}
\text{\em \small dim}(\kh i j L)
$$
is the Jones polynomial of $L$.
\end{prop}
\begin{proof}
 Let $P(L) = \Sigma_{i,j} (-1)^{i}q^j \text{\small dim}(\kh i j L)$ and suppose  $L$ is represented by a diagram $D$. The alternating sum of dimensions in a long exact sequence of vector spaces is always zero, so from the long exact sequence for a negative crossing we have that for each $j\in \bbZ$ the sum
$$
 \sum_{i} (-1)^{i}\text{\small dim}(  \kh i {j+1} {\! \ores \!}) 
 - \sum_{i} (-1)^{i}\text{\small dim}(\kh i {j} {\! \negx \!}) + \sum_{i}
 (-1)^{i}\text{\small dim}(\kh {i-\omega} {j-1-3\omega} {\! \uoreszero
   \!}) 
$$
is zero. Written in terms of the polynomial $P$ this becomes
$$
q^{-1}P(\ores)- P(\negx) + (-1)^\omega q^{1+3\omega} P(\uoreszero) = 0.
$$
Similarly, using the long exact sequence for a positive crossing (noting that $c=\omega +1$) we get 
$$
(-1)^\omega q^{5+3\omega}P(\uoreszero) - P(\posx) +  q P(\ores) = 0.
$$
Combining these gives
$$
q^{-2}P(\posx) - q^2P(\negx) + (q-q^{-1})P(\ores)=0
$$
which becomes the skein relation of the Jones polynomial when $q=-t^{\frac12}$. Since $P(\text{unknot}) = q + q^{-1} =  -(t^\frac12 + t^{-\frac12})$, the uniqueness of the Jones polynomial gives
$$
P(D)\mid_{q=-t^{\frac12} } = -(t^\frac12 + t^{-\frac12}) J(D)
$$
whence the result.
\end{proof}

\begin{rema}
Rasmussen has given tentative definition of what a knot homology theory should encompass
(somewhat different from the expectations given above) discussing both
Khovanov homology and Heegaard-Floer knot homology (Rasmussen, \newciteold
{0504045}). \end{rema}

\subsection{Reduced Khovanov homology}
There is a further piece of structure induced on Khovanov homology defined in the following way. The Khovanov homology of the unknot is a ring with unit courtesy  of the cobordisms
 \raisebox{-1.3mm}{\includegraphics[scale=0.03]{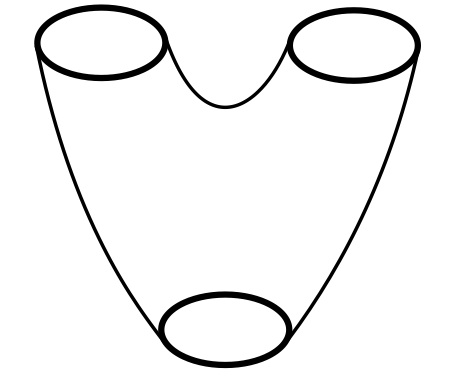}}  and
 \raisebox{3mm}{\includegraphics[scale=0.1, angle=180]{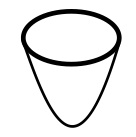}}  which
 induce multiplication
\begin{wrapfigure}{l}{3.5cm}
\includegraphics[scale=0.22]{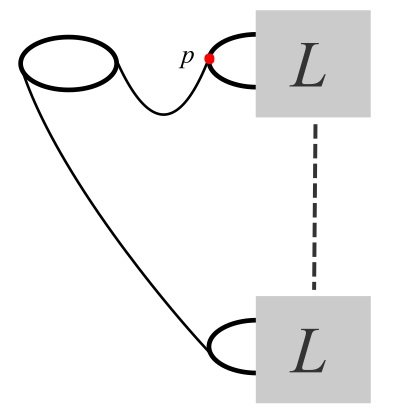}
\end{wrapfigure} 
 and unit  respectively.  The Khovanov homology of a link $L$ together with a chosen point $p$ is a module over this ring, using the link cobordism indicated.
A priori this module structure depends on the point $p$ and in particular on the component of $L$ to which $p$ belongs. Although it does not follow  from the existence theorem directly, for the version of Khovanov homology presented above (namely over $\bbF_2$), this structure does not depend on these choices. In fact more is true (again not immediate from the existence theorem) and the structure of $\kh ** L $ over $U^{**} = \kh**{\text{unknot}}$ can be described as follows: there exists a bigraded vector space $\rkh ** L$ with the property that
$$
\kh **L \cong \rkh **L \otimes U^{*,*}.
$$

\begin{theo} [Existence of reduced Khovanov homology] There exists a (covariant) functor
$$
\prerkh^{*,*} \colon \links \ra \vectftwo
$$
satisfying
\begin{enumerate}
\item If $\Sigma\colon L_1\ra L_2$ is an isotopy then $\prerkh (\Sigma)$ is an isomorphism,
\item $\rkh**{L_1\sqcup L_2} \cong \rkh ** {L_1} \otimes \rkh **
  {L_2}\otimes U^{*,*}$,
\item $\rkh**{\text{unknot}} = \bbF_2$ in bidegree $(0,0)$,
\item $\prerkh^{*,*}$ satisfies the same long exact sequence (with the
  same bigradings) written down previously for the unreduced case.
\end{enumerate}
\end{theo}

There is a question about which category of links we should
    be using here. A natural one would be links with a marked point
    and cobordisms with a marked line. If $L$ has an odd number of components then
  $\prerkh^{*,\text{odd}}$ is
  trivial and if it has an even number of components then
  $\prerkh^{*,\text{even}}$ is trivial. By a similar argument to what
  we saw previously we have
 $$
\sum_{i,j} (-1)^{i+j}t^{\frac{j}2} \text{dim}(\rkh i j L) = J(L).
$$

In order to compute Khovanov homology we should use our tools for that
purpose which are the two long exact sequences. 

\begin{exer}[beginners]
  Using the long exact sequences calculate the reduced Khovanov
  homology of the Hopf link, left and right trefoils, and the figure
  eight knot.
\end{exer}
\begin{exer}[experts]
  Find the first knot in the tables for which the reduced Khovanov
  homology can not be calculated using only the long exact sequences
  and calculations of the reduced Khovanov homology of knots and links
  occurring previously in the tables.
\end{exer}

Alternating links have particularly simple Khovanov homology (Lee, \newciteold{0201105}).

\begin{prop}\label{prop:alternating}
  For a non-split alternating link  $L$ the vector space
  $\rkh ij {L}$ is trivial unless$ j-2i$ is the signature of $L$. 
\end{prop}

As a corollary we note that for an alternating link  $\rkh** L$ is
completely determined by the Jones polynomial and signature.
Lee's result can also be proved using an approach to Khovanov homology
using spanning trees (Wehrli,  \newciteold{0409328}).

\begin{rema}
  The Khovanov homology of the unknot has more structure than that of
  a ring. There are also cobordisms  \raisebox{3.5mm}{\includegraphics[scale=0.03, angle=180]{pants}}  and
 \raisebox{-1.5mm}{\includegraphics[scale=0.1]{cup}}  which
 induce maps at the algebraic level. These maps and the ones above are subject to
 relations determined by the topology of surfaces. The upshot is that
 $U^{**}$ is a Frobenius algebra.
\end{rema}

\subsection{Integral Khovanov homology}
One can also define an integral version of the theory which has
 long exact sequences as above, but some changes are necessary. 
\begin{enumerate}
\item Functoriality is much trickier (see section \ref{subsec:func} below for references)
  \begin{itemize}
  \item up to sign $\pm 1$ everything works okay,
\item strict functoriality requires work.
  \end{itemize}

\item There is a reduced version but
  \begin{itemize}
  \item it is dependent on the component of the marked point,
\item the relationship to the unreduced theory is more complicated and
  is expressed via a long exact sequence,

$$
\xymatrix{
\ar[r]^(.25){\delta} &\irkh i{j+1}  {L,L_\alpha} \ar[r] & \ikh i{j}  L \ar[r] & \irkh {i}{j-1}  {L,L_\alpha} \ar[r]^(.75){\delta} &
}
$$
where $L_\alpha$ is a chosen component of $L$,
\item in this exact sequence the coboundary map $\delta$ is zero modulo 2.
  \end{itemize}
\end{enumerate}

The integral theory is related to the $\bbF_2$-version by a universal
coefficient theorem. There is a short exact sequence:
$$
\xymatrix{
0 \ar[r]  & \ikh i{j}  L \otimes \bbF_2 \ar[r] & \khkh ij {\bbF_2} L \ar[r] & \text{Tor}(\ikh {i+1}{j}  L , \bbF_2)  \ar[r] & 0 
}
$$
There is a similar universal coefficient theorem relating the two
reduced theories.

Any theory defined over the integers has a chance of revealing
interesting torsion phenomena. Unreduced integral Khovanov homology
has a lot of 2-torsion and much of this arises in the passage from
reduced to unreduced coming from that fact that  in the long exact
sequence relating the two theories  the coboundary map is zero mod 2. Correspondingly the reduced theory has much less 2-torsion.

\begin{prop}
 The reduced integral Khovanov homology of alternating links has no
 2-torsion.
\end{prop}
\begin{proof}
 Suppose that $L$ is non-split. Any 2-torsion in $\irkh ijL $ would contribute non-trivial homology in $ \rkhkh ij {\bbF_2} L$  via the leftmost group in the universal coefficient theorem for reduced theory and also in   $ \rkhkh {i-1}j {\bbF_2} L$ via the Tor group. This contradicts the conclusion of Proposition \ref{prop:alternating}, namely that there is only non-trivial homology when $j-2i = \text{signature}(L)$.
 \end{proof}

In general torsion is not very well understood. Calculations (by Alex Shumakovitch) show
\begin{itemize}
\item the simplest knot having 2-torsion in reduced homology has 13 crossings, for example 13n3663,
\item the simplest knot having odd torsion in unreduced homology is $T(5,6)$ which has a copy of $\bbZ /3$ and a copy of $\bbZ /5$,
\item the simplest knot having odd torsion in reduced homology is also $T(5,6)$ which has a copy of $\bbZ /3$,
\item some knots, e.g. $T(5,6)$, have odd torsion in unreduced homology which is not seen in the reduced theory, but the other way around is also possible: $T(7,8)$ has an odd torsion group in reduced that is not seen in unreduced.
\end{itemize}

Torus knots are a very interesting source of odd torsion and in fact
almost all odd torsion observed so far has been for torus knots. There is, however,
an example of a non-torus knot 5-braid which has 5-torsion (Przytycki
and Sazdanovi\'c, \newcite{1210.5254}). In general torsion remains quite a mystery.

\section{Constructing Khovanov homology}
The central combinatorial input in the construction of Khovanov
homology is a hypercube decorated by modules
known variously as ``the cube'', ``the cube of resolutions'' and
``the Khovanov cube'' and  it is constructed from a diagram
representing the link in question. This cube of resolutions is actually an example of something more general: it is a Boolean lattice
equipped with a local coefficient system and this is the point of view we take in these notes.

\subsection{Posets and local coefficient systems}
Let $\bbP$ be a partially ordered set and $R$ a ring. For Khovanov homology the posets of most importance are {\em Boolean lattices}: for a
\begin{wrapfigure}{r}{5.3cm}
\vspace{-5mm}
\includegraphics[scale=0.22]{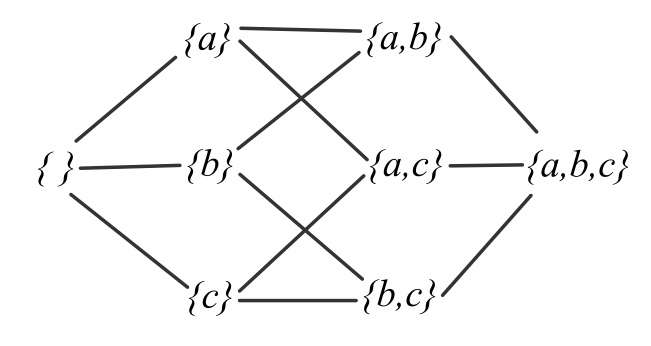}
\vspace{-7mm}
\end{wrapfigure} 
  set $S$, the  {\em Boolean lattice} on $S$ is the poset $\bbB (S)$ consisting of the set of subsets of $S$ ordered by inclusion.  The Hasse diagram of the Boolean lattice $\bbB (\{a,b,c\}) $ is shown here. The {\em Hasse diagram} of a poset $\bbP$ is the graph having vertices the elements of $\bbP$ and an edge \raisebox{-0.8mm}{\includegraphics[scale=0.16]{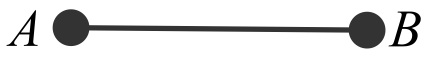}} if and only if $A < B$ and there is no $C$ such that $A<C<B$ (in this case say $B$ {\em
  covers} $A$). We will adopt the pictorial convention that if $A<B$ then
the Hasse diagram features $A$ to the left of $B$. 

A poset $\bbP$ can be regarded as a category with a unique morphism $A\ra B$ whenever $A\leq B$ and a {\em system of local coefficients} for $\bbP$ consists of a (covariant) functor $F\colon \bbP \ra \modR$. 

\begin{exem}\label{ex:kh}
 Let $D$ be an oriented link diagram 
and let $X_D$ be the set of crossings. In this example we will construct a local coefficient system on the boolean lattice $\bbB (X_D)$. We will define a functor $F_D\colon \bbB(X_D) \ra \modR$ first on objects and then on morphisms.
 \begin{itemize}
 \item {\em Objects:}
Let $A\subset X_D$. Near each crossing $c\in X_D$ do ``surgery'' as follows:

\begin{center}
\includegraphics[scale=0.25]{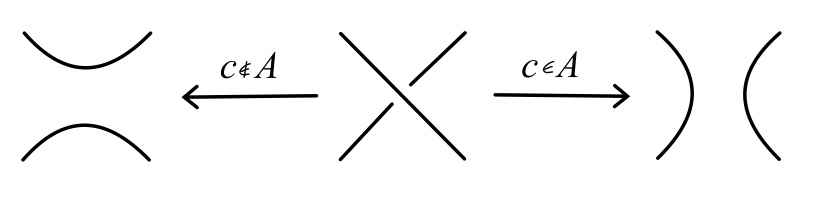}
\end{center} 

The result is a collection of closed circles in the plane and we refer to this as the {\em resolution associated to} $A$. The right-handed trefoil with crossing set $\{1,2,3\}$ has typical resolution as shown below (for the  subset $A=\{2\}$).  
\begin{center}
\includegraphics[scale=0.4]{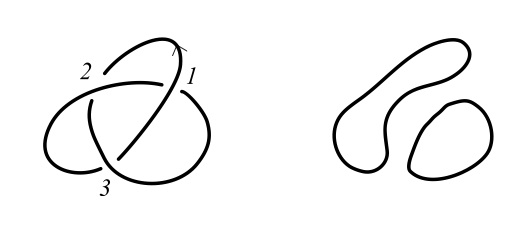}
\end{center} 
Now define $F_D(A)$ to be the truncated polynomial algebra with one generator for each component of the resolution associated to $A$.
$$
F_D(A) = P [x_\gamma \mid \gamma \text{ a component of the resolution associated to $A$}] / ( x_\gamma^2 = 0)
$$
\item  {\em Morphisms:} Suppose that $B$ covers $A$. We must define a map $F_D(A<B) \colon F_D(A) \ra F_D(B)$ which to simplify the notation we will denote by $d_{A,B}$.
By assumption $B$ has exactly one more element than $A$ and in a neighbourhood the additional crossing we see the following local change:

\begin{center}
\includegraphics[scale=0.28]{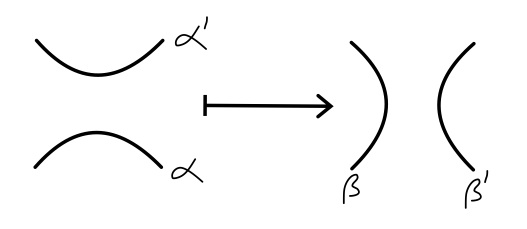}
\end{center} 

There are now two cases.
\begin{itemize}
\item if $\alpha\neq \alpha^\prime$ (in which case we also have $\beta=\beta^\prime$) we define $d_{A,B}$ to be the algebra map defined to be the identity on all generators apart from $x_\alpha$ and $ x_{\alpha^\prime}$ where 
$$
x_\alpha, x_{\alpha^\prime}\mapsto x_\beta, 
$$

\item  if $\alpha= \alpha^\prime$ (in which case we also have $\beta\neq\beta^\prime$) we define $d_{A,B}$ to be the {\em module} map given by
$$
\underline{x}  \mapsto  (x_\beta + x_{\beta^\prime}) \underline{x}  
$$
where $\underline{x} $ is a monomial in which $x_\alpha$ does not  appear, and
$$
 \underline{x} x_\alpha \mapsto \underline{x}x_\beta x_{\beta^\prime}.
$$
 Thus, in particular,  $1\mapsto x_\beta + x_{\beta^\prime}$ and $x_\alpha \mapsto  x_\beta x_{\beta^\prime}$.
\end{itemize}

Finally, if $A<B$ is an arbitrary morphism, decompose it as $A=A_0<A_1<\cdots <A_k=B$ where $A_{i+1}$ covers $A_i$ and define $F_D(A<B)$ to be the composition of the maps $d_{A_i,A_{i+1}}$ defined above. 
\end{itemize}

\begin{exer}
 Check that the definition for an arbitrary morphism is independent of the decomposition into one-step morphisms.
\end{exer}

\end{exem}

The example above is in fact central to the construction of Khovanov homology but there is one embellishment needed, namely that the  local coefficient system takes values in {\em graded} modules. Let $A\subset X_D$ and let $||A||$ denote the number of components in the resolution associated to $A$. The monomial $x_{\gamma_1}\ldots x_{\gamma_m}$ is defined to have grading
$|A| + ||A|| - 2m$. If, for example, $A$ has 5 elements and the
associated resolution has  two components then there are four
generators $1,x_1, x_2$ and $x_1x_2$ of degrees $7, 5, 5$ and $3$
respectively. With these gradings the functor $F_D\colon \bbB(X_D) \ra
\grmodR $ is graded in the sense that morphisms induce maps of graded modules of degree zero.

\subsection{Extracting information from local coefficient systems on Boolean lattices}
Let  $S$ be a set and $F\colon \bbB (S) \ra \grmodR$ a graded local
coefficient system (the assumption is that morphisms induce maps of
degree zero). We can now define a bigraded cochain complex $(\cC^*,d)$ with cochain groups
$$
\cC^{i,j} = \bigoplus_{\substack{A\subset S, |A|=i}} F^j(A).
$$
In order to define a differential $d\colon \cC^{i,j} \ra \cC^{{i+1},j}$ we use
the ``matrix elements'' $d_{A,B}\colon F(A) \ra F(B)$ where $A$ and
$B$ range over subsets of size $i$ and $i+1$ respectively and there is
such a matrix element whenever $B$ covers $A$. Explicitly, for $v\in
F^j(A)\subset \cC^{i,j}$
$$
d(v) = \sum_{\substack{B \text{ covers } A}} d_{A,B}(v).
$$

As it stands $d^2$ is zero only mod 2 but this can be rectified by
introducing a {\em signage}  function $\epsilon\colon \{\text{edges}\}
\ra \bbZ /2$ which satisfies $\epsilon(e_1) +\epsilon(e_2)
+\epsilon(e_3) +\epsilon(e_4)=1$ mod 2 whenever $e_1, \cdots, e_4$ are
the four edges of a square in the Hasse diagram.The definition above
is modified to read
$$
d(v) = \sum_{\substack{B \text{ covers } A}} (-1)^{\epsilon(A,B)}d_{A,B}(v).
$$

 This will give $d^2=0$ over any ring and a different choice of
 signage will give an isomorphic complex.

\begin{defi}
 Let $L$ be an oriented link and let $D$ be a diagram representing $L$
 having $n$ crossings of which $n_-$ are negative and  $n_+$ positive. Let $F_D \colon
 \bbB (X_D) \ra \grmodR$  be 
as defined in Example \ref{ex:kh}. Applying the above construction
gives a bigraded cochain complex $\cC^{*,*}(D)$.  The {\em Khovanov
  homology} of $L$ with coefficients in the ring $R$ is the (shifted)
homology of this complex:
$$
\prekh^{i,j} (D) = H^{i+n_-, j+2n_- - n_+}((\cC^{*,*}(D), d)
$$
\end{defi}

\begin{rema} 
The shifts by $n_-$ and $2n_- - n_+$ are global
  shifts needed in order to obtain an invariant (see Theorem
  \ref{thm:inv} below) in much the same way as the Kauffman bracket
  formulation of Jones polynomial requires an additional factor
  depending on the writhe of the diagram used.
\end{rema}
\begin{rema}
 For $A\subset X_D$, the monomial
 $x_{\gamma_1}\ldots x_{\gamma_m} \in F(A)$ defines a cochain in bidegree
$$
(|A| - n_-, |A| + ||A|| - 2m -2n_- + n_+)
$$ 
\end{rema}

\begin{exer}
  Read sections 3.1 and 3.2  of {\em On Khovanov's
    categorification of the Jones polynomial} (Bar-Natan, \newciteold{0201043}) and/or
  sections 2.1 and 2.2 of {\em Khovanov homology theories and their
    applications} (Shumakovitch, \newcite{1101.5614}) and marry the description of the construction of Khovanov homology with what is written above.
\end{exer}

This construction appears to depend on the diagram, but Khovanov's
first main result is that it doesn't (Khovanov, \newciteold{9908171}, see also
Bar-Natan, \newciteold{0201043}).
\begin{theo}\label{thm:inv}
 Up to isomorphism the definition above does not depend on the choice of diagram representing the link.
\end{theo}

\begin{rema}
  In fact one need not go all the way to taking homology: the cochain complex itself is an invariant up to homotopy equivalence of complexes.
\end{rema}

\begin{exer}
Using the construction above show that  $\prekh(L_1\sqcup L_2) \cong \prekh(L_1) \otimes \prekh(L_2)$.
\end{exer}

\begin{exer}
Show that if $L$ is presented by a diagram part of which is $ \uox$ then there is a short exact sequence of complexes
$$
\xymatrix{
0 \ar[r] & \cC^{*-1}(\uoresone) \ar[r] &  \cC^*(\uox) \ar[r] & \cC^*(\uoreszero) \ar[r] & 0
}
$$
Verify that the induced long exact sequence has gradings as presented
in Theorem \ref{thm:existence}.
\end{exer}

\subsection{Functoriality}\label{subsec:func}
The existence theorem asserts that there is a {\em functor} from a
category of links and  link cobordisms to modules. Courtesy of the
construction in the previous section we know how to define this
functor on links, but what is still needed is how to define module
maps associated to link cobordisms. 

\begin{exer}
 Find out how link cobordisms are represented by {\em movies} and how
 to associate maps in Khovanov homology to such things (by looking in the
 papers cited below for example).
\end{exer}

Because of the dependence on diagrams there are things to check. One can show that up to an over
all factor of $\pm 1$ there is no dependence of the maps on the diagrams
chosen. This is enough to give a functor over $\bbF_2$. The papers
showing functoriality up to $\pm 1$ are
\begin{itemize}
\item {\em An invariant of link cobordisms from Khovanov homology}
  (Jacobsson, \newciteold{0206303})
\end{itemize}
\begin{itemize}
\item {\em An invariant of tangle cobordisms} (Khovanov, \newciteold{0207264})
\end{itemize}
and
\begin{itemize}
\item {\em Khovanov's homology for tangles and cobordisms} ( Bar-Natan, \newciteold{0410495})
\end{itemize}

It is hard work to remove the innocent looking ``up to $\pm 1$'' and
something additional is needed to make it work. One approach is to using Bar-Natan's
local geometric point of view (see section \ref{sec:bnlocal} below)
\begin{itemize}
\item {\em Fixing the functoriality of Khovanov homology}  (Clark, Morrison and Walker, \newciteold{0701339})
\end{itemize}
which requires working over $\bbZ [i]$. A somewhat similar point of
view is developed in
\begin{itemize}
\item {\em An $sl(2)$ tangle homology and seamed cobordisms} (Caprau, \newcite{0707.3051})
\end{itemize}
A different construction working over $\bbZ$ is in 
\begin{itemize}
\item {\em An oriented model for Khovanov homology} (Blanchet,  \newcite{1405.7246})
\end{itemize}

\subsection{Aside for algebraic topologists: another extraction technique}
There is another, more abstract, way of extracting information from
the cube. To motivate this kind of approach think about the
definitions of group cohomology where one can either define an explicit cochain complex using the bar
resolution or use derived functors. Each approach has
its uses and if the definition is taken to be the explicit complex
then the derived functors approach becomes an 
``interpretation'', but if the definition is in terms of derived
functors then the explicit complex becomes a
``calculation''.

 There is a way of defining cohomology of posets
equipped with coefficient systems by using the right derived functors
of the inverse limit. With a
small modification to the underlying Boolean lattice, this gives an alternative way of getting
Khovanov homology. Let $Q$ be the poset formed from
$\bbB (X_D)$ by the addition of a second minimal element. Extend the
functor $F_D$ to $Q$ by sending this new element to the trivial
group. Khovanov homology can be interpreted as the right derived functors
of the inverse limit functor (Everitt and Turner, \newcite{1112.3460}): 
$$
\prekh (D) \cong  R^* \invlim_Q F_D
$$

\section{Odd Khovanov homology}
The construction of Khovanov homology makes no demands on the order of the
circles appearing in a resolution. At the algebraic level this is
reflected in the polynomial variables commute among
themselves.  If one could impose a local ordering of strands near
crossings then one might hope that this commutativity requirement
could be removed. The subject of {\em odd} Khovanov homology is one
approach to achieving this. The defining paper is
\begin{itemize}
\item   {\em Odd Khovanov homology} (Ozsv\'ath, Rasmussen and Szab\'o, \newcite{0710.4300})
\end{itemize}
and there is also a nice expository article
with many calculations 
\begin{itemize}
\item {\em Patterns in odd Khovanov homology} (Shumakovitch, \newcite{1101.5607})
\end{itemize}

The construction of odd Khovanov homology is a refinement of the construction of (ordinary) Khovanov homology given in the last section.
Let $D$ be an oriented link diagram and let $X_D$ be the set of crossings. We will construct a local coefficient system  $F^{\text{odd}}_D\colon \bbB (X_D)\ra \modR$  on the boolean lattice $\bbB (X_D)$.
 \begin{itemize}
 \item {\em Objects:}
Let $A\subset X_D$. Near each crossing $c\in X_D$ replace the crossing according to the following two rules  

\begin{center}
\includegraphics[scale=0.25]{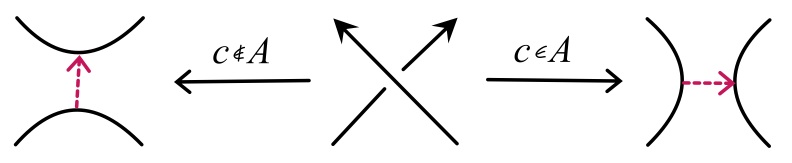}
\end{center} 

\begin{center}
\includegraphics[scale=0.25]{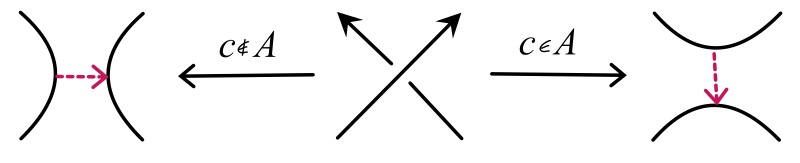}
\end{center} 

The result is a collection of closed circles in the plane with a number of additional dotted arrows and we refer to this as the {\em odd resolution associated to} $A$. The right-handed trefoil with crossing set $\{1,2,3\}$ has typical resolution as shown here (for the  subset $A=\{1,3\}$).  
\begin{center}
\includegraphics[scale=0.4]{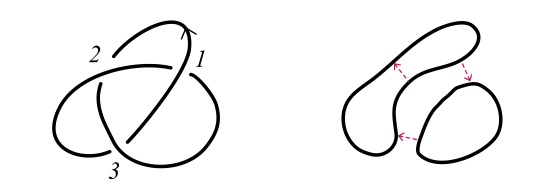}
\end{center} 
Note that if presented with a local piece around a crossing for which the two strands are from different components we may order these by the decree: tail before head. This does not give a global ordering on the circles. 

Now define $F^{\text{odd}}_D(A)$ to be the exterior algebra with one generator for each component of the odd resolution associated to $A$.
$$
F^{\text{odd}}_D(A) = \Lambda_A = \Lambda [x_\gamma \mid \gamma \text{ a component of the odd resolution associated to $A$}]
$$
\item  {\em Morphisms:} Suppose that $B$ covers $A$. We must define a map $d^{\text{odd}}_{A,B} \colon F^{\text{odd}}_D(A) \ra F^{\text{odd}}_D(B)$. 
By assumption $B$ has exactly one more element than $A$ and in a neighbourhood the additional crossing we see the following local change:

\begin{center}
\includegraphics[scale=0.28]{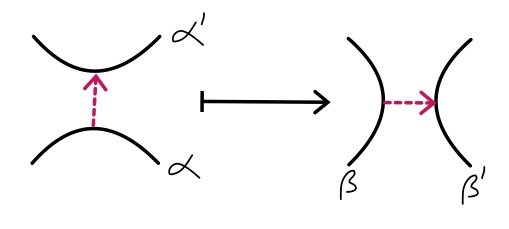}
\end{center} 

There are now two cases.
\begin{itemize}
\item if $\alpha\neq \alpha^\prime$ (in which case we also have $\beta=\beta^\prime$) we define $d_{A,B}$ to be the algebra map given by
$$
1\mapsto 1, \;\;\;\;\;\; x_\alpha, x_{\alpha^\prime} \mapsto x_\beta, \;\;\;\;\;\; x_\gamma \mapsto x_\gamma \text{ for $\gamma \neq \alpha, \alpha^ \prime$}
$$
\item  if $\alpha= \alpha^\prime$ (in which case we also have $\beta\neq\beta^\prime$) we define $d_{A,B}$ to be the {\em module} map given by
$$
x_\alpha \wedge v \mapsto x_\beta\wedge x_{\beta^\prime} \wedge v, \;\;\;\;\;\; v \mapsto  (x_\beta -x_{\beta^\prime})\wedge v
$$
where $x_\alpha$ is assumed not to appear in $v$. Thus $1\mapsto x_\beta -x_{\beta^\prime}$ and $x_\alpha \mapsto  x_\beta\wedge x_{\beta^\prime}$.
\end{itemize}
The first of these makes no use of the local ordering, but in the
second the asymmetry is very clear.
\end{itemize}

\begin{exer}\label{ex:modtwoodd}
 Check that if the underlying ring is the
field $\bbF_2$ then the exterior algebra is isomorphic to the truncated
polynomial algebra and the matrix element maps $d^{\text{odd}}_{A,B}$ agree with the ones used  in
construction of  (ordinary) Khovanov homology. 
\end{exer}

For ordinary Khovanov homology the construction  gives a functor $\bbB
(X_D) \ra \modR$ without further trouble. Or, put differently, the
square faces of the cube commute. (Immediately afterwards a sign
assignment is made, but that is to turn commuting squares into
anti-commuting ones which is only necessary because of the particular extraction technique used to obtain a complex out of the functor.) Here, for odd Khovanov homology things are not so simple and there is not obviously functor $\bbB (X_D) \ra \modR$; some squares commute, others anti-commute and others still produce maps which are zero. After a fair bit of digging into the possible cases Ozsv\'ath-Rasmussen-Szab\'o prove:

\begin{prop}
  There exists a signage making all squares commute. 
\end{prop}

This gives a functor $F^{\text{odd}}_D\colon \bbB (X_D)\ra \modR$ and by
one of the extraction techniques discussed previously this yields a
complex whose homology defines {\em odd Khovanov homology}, denoted
 $\khodd ** (D;R)$.

 \begin{rema}
  The bigrading is as follows: for $A\subset X$, the $m$-form $x_{\gamma_1} \wedge \cdots \wedge x_{\gamma_m} \in
  \Lambda^m\subset \Lambda_A$ defines a cochain with bigrading
$$
(|A| -n_-, |A| + || A|| - 2m  - 2n_- +n_+ )
$$ 
where as before $||A||$ is the number of circles in the odd resolution defined by $A$.
 \end{rema}

Odd Khovanov homology shares many properties of ordinary Khovanov homology, but there are some crucial differences. Here is a summary of some of its properties:
\begin{itemize}
  \item there are skein long exact sequences precisely as for ordinary
    Khovanov homology (with the same indices),
\item the Jones polynomial is obtained as
$$
\sum_{i,j} (-1)^i q^j \text{dim}(\khodd i j (L))|_{q=-t^{\frac{1}{2}}}
= -(t^{\frac{1}{2}} + t^{-\frac{1}{2}}) J(D),
$$
\item there is a reduced version, $\khoddr **$, satisfying
$$
\khodd i j \cong \khoddr i {j+1} \oplus \khoddr i {j-1}
$$
and which does not depend on the component of the base-point. So
$\prekhodd$ stands in the same relationship to $\prekhoddr$ as
$\prekh_{\bbF_2}$ to $\prerkh_{\bbF_2}$ which is very different to
the relationship between $\preikh$ and $\preirkh$,
\item over $\bbF_2$ odd
and ordinary Khovanov homology coincide (reduced and unreduced); this is  courtesy of Exercise \ref{ex:modtwoodd},
\item $\khoddr ** (\text{alternating})\cong \prerkh^{**}
  (\text{alternating})$ but in general $\prerkh$ neither determines or
  is determined by $\prekhoddr$.
\end{itemize}

\begin{rema}
 The is a spectral sequence with $E_2$-page $\prekh_{\bbF_2}(L^{!})$
    converging to the Heegaard-Floer homology of the
    double branched cover branched along $L$ (Ozsv\'ath and Szab\'o, \newciteold{0309170} ). 
    To lift this
    integrally the correct theory to put at $E_2$ is (conjecturally)
    odd integral Khovanov homology.
 Indeed this was one of the
    motivations for the invention of odd Khovanov homology. 
\end{rema}

\begin{rema}
 There are other interesting spectral sequences featuring odd
 Khovanov homology at the $E_2$-page.  There is
  one starting with odd Khovanov homology 
and converging to an
  integral version of a theory made by Szabo (Beier, \newcite{1205.2256}).
Another starts with odd Khovanov homology
  and converging to the framed instanton homology of the double
  cover (Scaduto, \newcite{1401.2093}  ). 
\end{rema}

\begin{rema}
 On seeing a typical odd resolution it is tempting to re-draw it as a graph whose
  vertices are the circles and whose directed edges are the dotted arrows. There is
  a description of odd Khovanov homology in  terms of {\em
    arrow graphs} (Bloom, \newcite{0903.3746}).
\end{rema}

\section{Tangles}
We now return to (ordinary) Khovanov homology. The topology of the resolutions of a  link diagram requires knowledge of the whole diagram and this is used in the construction of Khovanov homology (circles fuse or split depending on global information). None the less diagrams are made up of more basic pieces, namely tangles, and so it is natural to ask if Khovanov homology may be defined more locally. The difficulty is that while piecing together geometric data is easy, doing the same with algebraic data is never so simple.

\subsection{Khovanov's approach}
The first approach is due to Khovanov who studies $(m,n)$-tangles  such as the one shown here (Khovanov, \newciteold{0103190}).
\begin{center}
\includegraphics[scale=0.14]{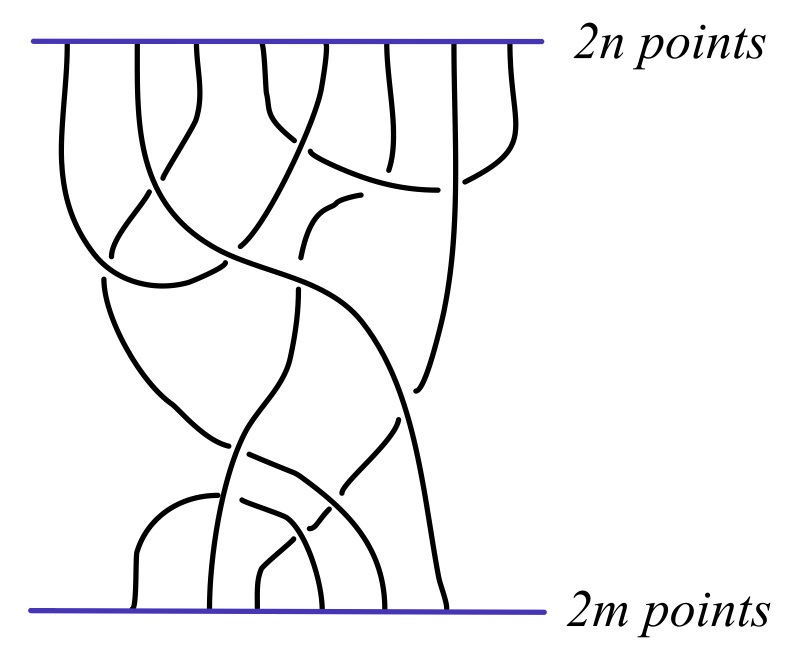}
\end{center}
 For such a diagram $T$ there is a cube of resolutions (in our language: a local coefficient system   on the Boolean lattice on crossings) as before. To $A\subset X_T$ one
 associates 
$$
M_A = \bigoplus P[x_\gamma \mid \gamma \text{ a circle }]/ (x^2_\gamma =0)
$$
where the direct sum is over all tangle closures.
Each $M_A$  is an $(H^m,H^n)$-bimodule where $\{H^i\}$ is a certain family of
rings (with elements the top-bottom closures of a set of $2i$ points; these rings can also be related to parabolic category $\mathscr O$ (Stroppel, \newciteold{0608234})). By the usual extraction of a complex from a cube this yields a
complex of bi-modules $\cC (T)$. When $m=n=0$ one recovers the usual
Khovanov complex. Isotopic tangles produce complexes that are homotopy equivalent and the construction is functorial (up to $\pm 1$) with respect to tangle cobordisms (Khovanov, \newciteold{0207264} ). 

The key new property is that by using the bi-module structure tangle composition can be captured algebraically.

\begin{prop} Let $T_1$ be a $(m,n)$-tangle and $T_2$ a $(k,m)$-tangle. Then,
  $$\cC (\raisebox{-10.7mm}{\includegraphics[scale=0.22]{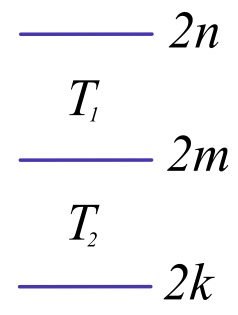}}) \simeq \cC (T_2) \otimes_{H^m} \cC(T_1)$$
\end{prop}

\subsection{Bar-Natan's approach}\label{sec:bnlocal}
A  different approach to locality is due to Bar-Natan as his viewpoint
has turned out to be very influential (Bar-Natan, \newciteold{0410495})  . The functor $F_D\colon \bbB
(X_D) \ra \grmodR$ comes from  two step process: firstly make 
resolutions (which are geometric objects) and secondly associate to to
these some algebraic data. The first step can be re-cast as  a functor
from $\bbB (X_D)$ to a cobordism category and the second step consists
of applying a 1+1-dimensional TQFT to the first step.  Bar-Natan's
central idea is to work with the ``geometric'' functor (or {\em cube}) as long as possible delaying the application of the TQFT.
$$
\xymatrix{
F_D\colon \bbB (X_D) \ar[rrr] &&& \text{Cob}_{1+1} \ar[rrr]^{\text{TQFT}} &&& \grmodR
}
$$
Distilling the essential operations used to construct a cochain
complex from the functor $F_D$ one sees that we needed to 1) take direct
sums of vector spaces (in the step often referred to as ``flattening
the cube''), and 2) assemble a linear map out of the matrix elements
which involved taking linear combinations of maps between vector
spaces. In order to delay the passage to the algebra and to build some
notion of ``complex'' in the setting of a cobordism category we need
some equivalent of these two operations.
What is done is to replace direct sum by the operation of taking
formal combinations of objects (closed 1-manifolds) and allowing
linear combinations of cobordisms. A typical morphism will be a matrix
of formal linear combinations of cobordisms. In this way it is
possible to define a ``formal'' complex $\leftsq D \rightsq $
associated to $D$.  It is no longer possible to take homology of such
formal complexes because we are working in a non-abelian category (the
kernel of a linear combination of cobordisms makes no sense, for
example) but one still has the notion of homotopy equivalence of
formal complexes and indeed if $D\sim D^\prime$ then $\leftsq D
\rightsq \simeq \leftsq D^\prime \rightsq$.

This approach works perfectly well for tangles too. Given a tangle $T$
of the type shown below a resolution will typically involved
1-manifolds with- and without-
\begin{wrapfigure}{l}{3.9cm}
\vspace{-5mm}
\includegraphics[scale=0.19]{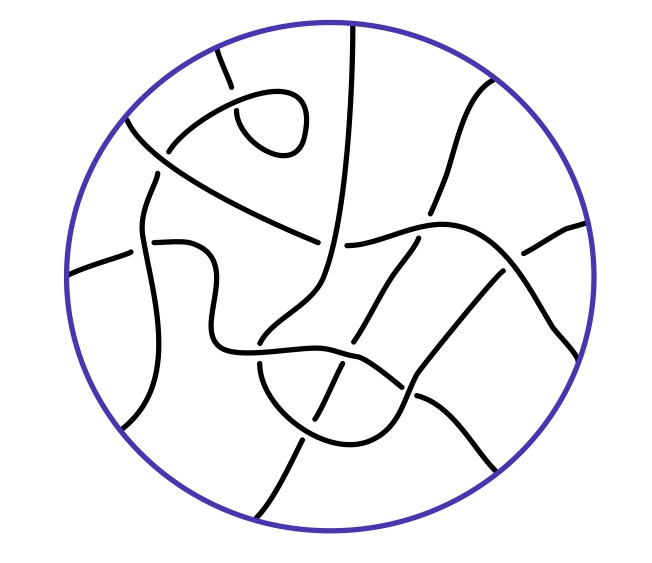}
\vspace{-9mm}
\end{wrapfigure}  
 boundary  and the cobordism category must
be adapted appropriately but a formal complex $\leftsq T \rightsq$ may
be constructed as above. Things are as they should be because given
isotopic tangles $T_1$ and $T_2$ then there is a equivalence of formal
complexes $\leftsq T_1 \rightsq \simeq \leftsq T_2 \rightsq$. Moreover
this construction is functorial (up to $\pm 1$) with respect to tangle
cobordisms.

But now comes the beauty of this approach: by insisting on staying on
 the geometric side of the street for so long, the composition of
 tangles is accurately reflected at the level of formal complexes as
 well. The combinatorics of tangle composition is captured by the
 notion of a {\em planar algebra}:  to each $d$-input arc diagram $D$
 like the one shown here 
\begin{center}
\includegraphics[scale=0.28]{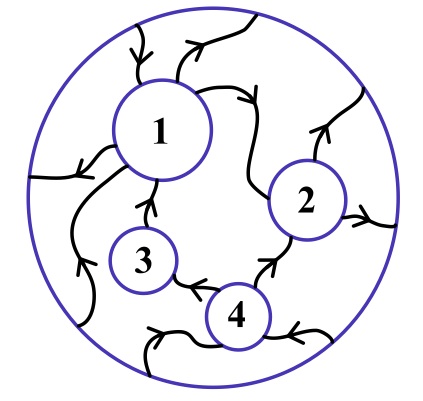}
\end{center}

\noindent
there is an operation
$$
D\colon \tang \times \cdots \times \tang \ra \tang
$$
defined by plugging the holes. For example
\begin{center}
\includegraphics[scale=0.28]{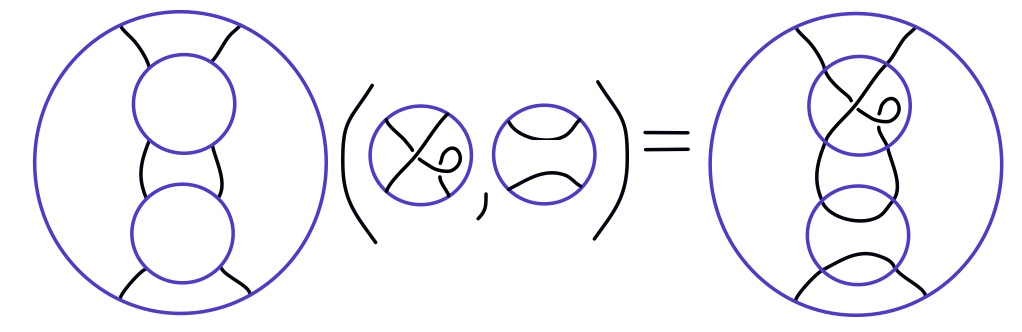}
\end{center} 

These operations are subject to various composition criteria that make up the structure of a planar algebra. The category $\tang$ is very naturally a planar algebra, but other categories may admit the structure of a planar algebra too - all that is needed is operations of the type above. The category of formal complexes above (the one in which $\leftsq T \rightsq$ lives) is an example - for the details of the construction you should read Bar-Natan's paper. 

\begin{prop}
 The construction sending a tangle $T$ to the associated formal complex $\leftsq T \rightsq$ respects the planar algebra structures defined on tangles and formal complexes. (In other words $\leftsq - \rightsq$ is a morphism of planar algebras)
\end{prop}

Using the planar algebra structure all tangles can be built out of 
single crossings. What the proposition is telling us is that the
same is true of the formal complex: it is enough to specify $\leftsq -
\rightsq$ on single crossings and the rest comes from the planar algebra structure. 

The next question to ask about this approach is: how does the planar algebra structure interact with link cobordisms? What is needed is an extension of the notion of planar algebra to the situation where there are morphisms between the planar algebra constituents. The name given to the appropriate structure is a {\em canopolis}. Again read Bar-Natan's paper for details.  Working with this local approach makes far more digestible the proofs of invariance and functoriality.

\begin{rema}
    If one wishes to apply a TQFT to get something
algebraic out of Bar-Natan's geometric complex one needs  something slightly different capable of handling manifolds with boundary. The appropriate
thing is an open-closed TQFT (Lauda and Pfeiffer, \newciteold{0606331}). The question of algebraic gluing of tangle components has also been  studied (Roberts, \newcite{1304.0463}  )  where inspiration is drawn from bordered Heegaard Floer homology and the skein module of tangles in the context of Khovanov homology (Asaeda, Przytycki and Sikora,  \newciteold{0410238} ).
\end{rema}

\begin{rema}
The complex constructed above can be simplified at an early stage by  a technique called {\em de-looping} (Bar-Natan,
\newciteold{0606318}).  Though very different in approach the de-looped
  complex is closely related to the one used by Viro in his description of Khovanov homology (Viro, \newciteold{0202199} ). 
\end{rema}

\begin{rema}
 In order to place odd Khovanov homology into a Bar-Natan-like geometric framework it is necessary to enrich the theory by working with 2-categories (Putyra,  \newcite{1310.1895}; Beliakova and Wagner \newcite{0910.5050} ). 
\end{rema}

\section{Variants}
In the definition of Khovanov homology we assigned a truncated polynomial algebra to a given resolution with one variable for 
each component of the resolution. It is possible to take the quotient by other ideals and still obtain a link homology theory. 
In fact for $h,t\in R$ a functor $F_D^{h,t}\colon \bbB(X_D) \ra \modR $ may be defined by
$$
F_D^{h,t}(A) = P [x_\gamma \mid \gamma \text{ a component of the resolution associated to $A$}] / ( x_\gamma^2 = t + h x_\gamma)
$$
with maps defined in a similar way to previously:
\begin{itemize}
\item if $\alpha\neq \alpha^\prime$
we define $d_{A,B}$ to be the algebra map defined to be the identity on all generators apart from $x_\alpha$ and $ x_{\alpha^\prime}$ where 
$$
x_\alpha, x_{\alpha^\prime}\mapsto x_\beta, 
$$
\item  if $\alpha= \alpha^\prime$ 
we define $d_{A,B}$ to be the {\em module} map given by
$$
\underline{x}  \mapsto  (x_\beta + x_{\beta^\prime} -h) \underline{x}  
$$
where $\underline{x} $ is a monomial in which $x_\alpha$ does not  appear, and
$$
 \underline{x} x_\alpha \mapsto \underline{x}x_\beta x_{\beta^\prime}
 + t.
$$
 Thus, in particular,  $1\mapsto x_\beta + x_{\beta^\prime}-h$ and $x_\alpha \mapsto  x_\beta x_{\beta^\prime}+t$.
\end{itemize}
In each case this gives rise to a link homology theory
(Khovanov, \newciteold{0411447}; Naot \newciteold{0603347}{).} To ensure a bigraded
theory the ground ring must also be graded and contain $h$ and $t$ of
degree $-2$ and $-4$ respectively. For Khovanov homology we  take
$t=h=0$ and the bigrading of the ground ring can be concentrated in
degree zero unproblematically.

\subsection{Lee Theory}
The first variant of Khovanov homology to appear was the case $h=0$
and $t=1$ working over $\bbQ$ (Lee, \newciteold{0201105} ). 
The ring $\bbQ$ is ungraded which means that Lee's theory is a singly graded theory. 
The two most important facts about this  theory are:
\begin{itemize}
\item it can be completely calculated for all links in terms of linking numbers
\item there is a filtration on the chain complex 
\end{itemize}

For calculation, Lee proves the following. (There is another proof this using
Bar-Natan's local theory (Bar-Natan and Morisson, \newciteold{0606542})).
\begin{theo}
  Let $K$ be a knot. Then
$$
\lee i (K) \cong 
\begin{cases}
  \bbQ \oplus \bbQ & i=0\\
0 & \text{ else.}
\end{cases}
$$
Let $L$ be a two component link. Then
$$
\lee i (K) \cong 
\begin{cases}
  \bbQ \oplus \bbQ & i=0, \text{ or } lk(\text{the two components})\\
0 & \text{ else.}
\end{cases}
$$
In general, for a $k$ component link
$
\sum \text{dim}(\lee i (L)) = 2^k
$
and there is a formula for the degrees of the generators in terms of
linking numbers.
\end{theo}

The filtration leads to a spectral sequence (implicit in Lee's paper, made
explicit in (Rasmussen, \newciteold{0402131}).
\begin{theo}
  Let $L$ be a link and $\gamma$ its number of components modulo two. There exists a spectral sequence, the {\em Lee-Rasmussen spectral
    sequence}, which has the form
$$
E_2^{p,q} = \prekh_\bbQ^{p+q, 2p+\gamma} \Longrightarrow \lee * (L).
$$
The differentials have the form $d_r\colon E_r^{p,q} \ra
E_r^{p+q,q-r-1}$. Moreover, each page of the spectral sequence is a
link invariant.
\end{theo}

\begin{rema}
  If one re-grades so that the differentials are expressed in
    terms of the gradings of Khovanov homology (rather than the
    pages of the spectral sequence) the differential $d_r$ is zero for
    $r$ odd and has
    bigrading $(1,2r)$ when $r$ is even.
\end{rema}

\begin{rema} In all known examples this spectral sequence (over $\bbQ$) collapses at the
  $E_2$-page. It is still an open question as to whether this is
  always the case or not.  
 \end{rema}

The utility of this spectral sequence is that it puts considerable
restrictions on the allowable shape of Khovanov homology. As an
example consider an attempted calculation of the rational (unreduced) Khovanov
homology of the right-handed trefoil only using the 
\begin{wrapfigure}{l}{6.5cm}
\includegraphics[scale=0.13]{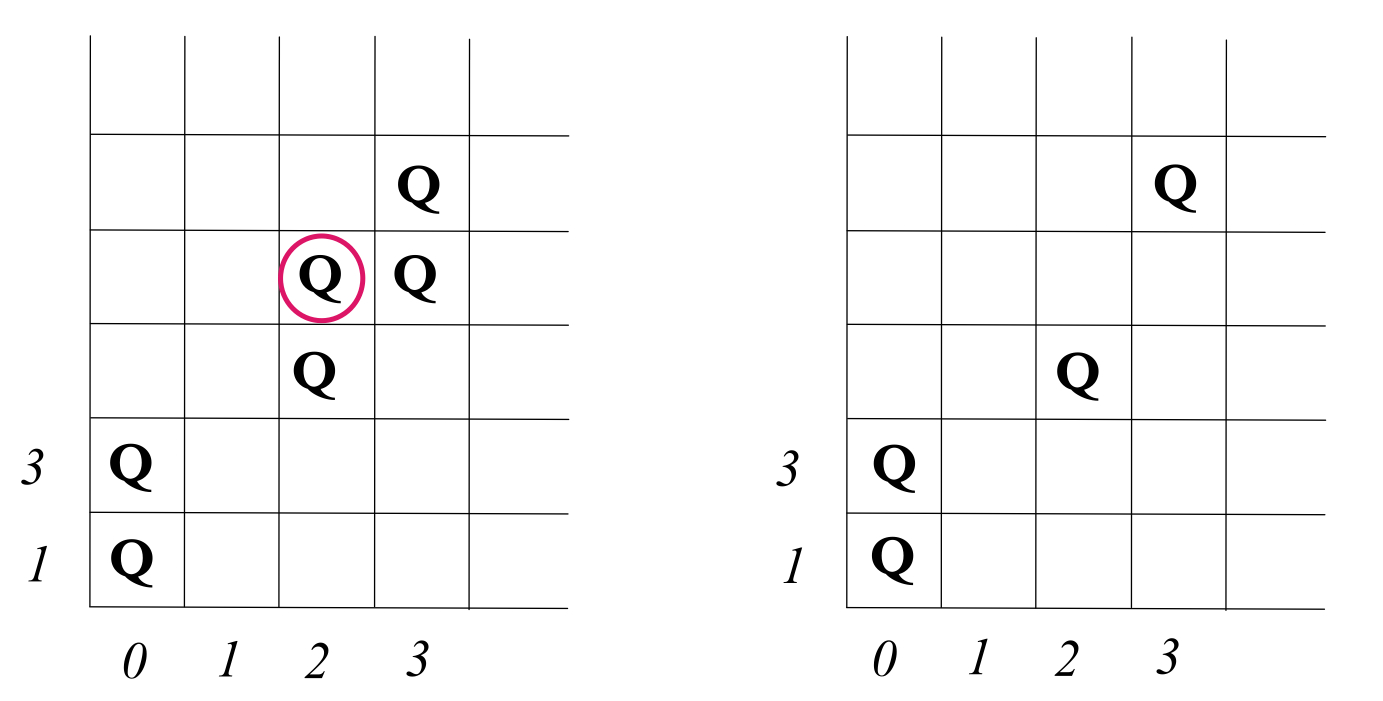}
\end{wrapfigure}  
skein long exact
sequences. At
some point you will find that you need additional information (some
boundary map may or may not be zero and you have no way of telling
without some further input).  You can conclude that the Khovanov homology must
be one of the  two possibilities shown here.
The existence of the Lee-Rasmussen spectral sequence tells you that
the correct answer is on the right: the two generators that survive to
the $E_\infty$-page of the spectral sequence are the two in
homological degree
zero and all the others must be killed by differentials; if the
Khovanov homology were as given on the left, then a quick look at the
degrees of the differentials shows that the generator in bi-degree (2,7)
could never be killed, giving a contradiction.

\begin{rema}
Over other rings Lee theory behaves as follows:
\begin{enumerate}
\item over $\bbF_p$ for $p$ odd it behaves as Lee theory over $\bbQ$
  (i.e. it is ``degenerate'') and the proof of Lee's theorem works verbatim.
\item over $\bbF_2$ a change of variables shows that it is isomorphic
  (as an ungraded theory) to $\bbF_2$
  Khovanov homology
\item over $\bbZ$ it has a free part of rank $2^{\text{no. of
      components}}$, no odd torsion, but a considerable amount of 2-torsion.
\end{enumerate}
\end{rema}

\begin{rema}
Over $\bbQ$ the above family (parametrized by $h$ and $t$) produces
only two isomorphism classes of theories: when $h^2 +4t=0$ the theory is
isomorphic to rational Khovanov homology and otherwise it is isomorphic to rational Lee
theory (Mackaay, Turner and Vaz, \newciteold{0509692}).   
\end{rema}

\subsection{Bar-Natan Theory}
Another interesting case is to take $t=0$ and $h=1$ (Bar-Natan,
\newciteold{0410495}).
 This theory is quite similar to Lee theory in the sense that
it is a
``degenerate'' theory requiring only linking numbers for a full
calculation and it is filtered with an attendant Lee-Rasmussen type spectral
sequence (Turner, \newciteold{0411225}). There
are, however, some differences (which possibly make it a better theory than Lee
theory): the integral version also degenerates and there is a reduced
version with a reduced Lee-Rasmussen type spectral sequence.

\section{Generalisations: $sl(N)$-homology and HOMFLYPT-homology}
There is a general procedure, due to Witten and Reshetikhin-Turaev,
for the construction of (quantum) link invariants using the
representation theory of quantum groups as input. Starting with a
simple Lie algebra $\frak g$, link components are labelled with
irreducible representations of the quantum group $U_q(\frak g)$ to
produce a link invariant. From this point of view the Jones polynomial
arises from the two dimensional representation when $\frak g = sl(2)$. 

An important and natural question is: are there
link homology theories associated to other Lie algebras which generalise
 Khovanov homology in some appropriate sense?

\subsection{ Khovanov-Rozansky $sl(N)$-homology}
An obvious place to start is $\frak g = sl(N)$; the case $N=2$ is
already done and the analogue of the Jones polynomial, the
$sl(N)$-polynomial has been extensively studied.

A very nice summary is given in
\begin{itemize}
\item {\em Khovanov-Rozansky homology of two-bridge
  knots and links } (Rasmussen, \newciteold{0508510})
\end{itemize}
and the details are contained in the original paper:  
\begin{itemize}
\item {\em Matrix factorizations and link homology} (Khovanov and
  Rozansky, \newciteold{0401268}).
\end{itemize}

\begin{theo}[Existence of $sl(N)$-homology] There exists a
  (covariant) projective functor
$$
\kr ** \colon \links \ra \vectq
$$
satisfying
\begin{enumerate}
\item If $\Sigma\colon L_1\ra L_2$ is an isotopy then $KR_N(\Sigma)$ is an isomorphism.
\item $\kr **(L_1\sqcup L_2) \cong \kr **(L_1) \otimes \kr ** (L_2)$.
\item $\kr i j(\text{unknot}) =
  \begin{cases}
  \bbQ & i= 0 \text{ and } j=2k-N-1 (k=1, \ldots , N)\\
0 & \text{else}  
  \end{cases}
$
\item There are long exact sequences:
$$
\hspace{-0.5cm}
\xymatrix@=12pt{
\ar[r]^(.2){\delta} & \kr {i-1}{j+N}  \sres \ar[r] & \kr i{j}  \negx \ar[r] & \kr {i}{j+N-1}  \ores \ar[r]^(0.8)\delta &
}
$$
$$
\hspace{-0.5cm}
\xymatrix@=12pt{
\ar[r]^(.2){\delta} & \kr {i}{j-N+1}  \ores \ar[r] & \kr i{j}  \posx \ar[r] & \kr {i+1}{j-N}  \sres \ar[r]^(0.8)\delta &
}
$$
\end{enumerate}
\end{theo}

Immediately we see there is something fishy with this: the long
exact sequences feature
\begin{wrapfigure}{l}{3cm}
\includegraphics[scale=0.2]{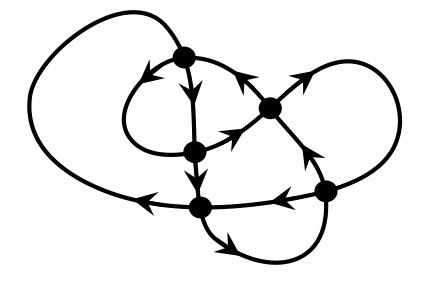}
\end{wrapfigure}  
  an as yet undefined object. In fact $\kr **$
assigns a bigraded vector space to each {\em singular link diagram}
(where crossing of the form $\posx$, $\negx$ and $\sres$ are
allowed). Up to isomorphism this assignment is invariant on deforming the diagram by
Reidemeister moves away from singularities. The situation is not quite
as good as for Khovanov homology because even with perfect information
about the long exact sequences, the basic normalising set of object
consists not of one single simple object (the unknot) but an infinity of
4-valent planar graphs such as the one shown here.

\begin{rema}
  If $\Sigma$ is a cobordism then $KR_N(\Sigma)$ has bi-degree $(0, (1-N)\chi(\Sigma))$.
\end{rema}

\begin{exer}
  Attempt a computation of $\kr ** (\text{Hopf link})$ from the
  existence theorem, carefully observing why this is much harder than
  when attempting the same computation for Khovanov homology.
\end{exer}

\begin{prop}
   Let $P_N (D) = \Sigma_{i,j} (-1)^{i}q^j \text{dim}(\kr i j (L))$ . We have
$$
q^{-N} P_N(\posx) - q^NP_N(\negx) + (q-q^{-1})P_N(\ores) =0
$$
and
$$
P_N(\text{unknot}) = \frac{q^N - q^{-N}}{q-q^{-1}}= \sum_{k=1}^N q^{2k-N-1}
$$
\end{prop}

We recognise these  two properties as the ones
characterising the $sl(N)$-polynomial showing that $P_N$ {\em is} the
$sl(N)$-polynomial.

\begin{exer}
  Prove this proposition using the long exact sequences and the fact
  that the alternating sum of dimensions in a long exact sequence is
  always zero.
\end{exer}

The construction of $sl(N)$-homology (and the proof of the existence
theorem above) proceeds once again by defining a local coefficient
system on a the Boolean lattice of crossings of a link diagram (a
decorated ``cube'' if you prefer that language). 
As before this begins by constructing a {\em resolution} for each
subset $A$ of the set of crossings $X_D$. Each resolution
is a planar singular graph and the rules for its construction
are:

$$
\xymatrix@=60pt{
\raisebox{-5.3mm}{\includegraphics[scale=0.16]{ores}}  & \raisebox{-5.3mm}{\includegraphics[scale=0.16]{negx}} \ar[l]_{c\notin A} \ar[r]^{c\in A} & \raisebox{-5.3mm}{\includegraphics[scale=0.16]{sres}} 
}
$$

$$
\xymatrix@=60pt{
\raisebox{-5.3mm}{\includegraphics[scale=0.16]{sres}}  & \raisebox{-5.3mm}{\includegraphics[scale=0.16]{posx}} \ar[l]_{c\notin A} \ar[r]^{c\in A} & \raisebox{-5.3mm}{\includegraphics[scale=0.16]{ores}} 
}
$$

If $B$ covers $A$ (it contains exactly one more crossing) then the corresponding resolutions are identical except in a small neighbourhood of the additional crossing where one of the following two local changes is seen: $\xymatrix{\ores \ar[r] &
  \sres}$ or  $\xymatrix{\sres \ar[r] & \ores}$. What is now needed is
a way of associating a module to each resolution and maps corresponding to cover relations (to cube
edges) giving a functor $F_{D,N}\colon \bbB (X_D) \ra \grvectQ$ from which a
complex and its homology can be extracted as before. For this to be
worth anything it must result in a link invariant and therein, of
course, lies the difficulty. 

Khovanov and Rozansky employ {\em matrix factorizations} in order to
carry this out. There are some guiding principles coming from the
description of the $sl(N)$-polynomial given by H. Murakami, Ohtsuki
and Yamada who describe it in terms of certain graphs which suitably
interpreted are the ones considered here. One may think of their
construction as associating a Laurent polynomial $\moy (\gamma)$ to each
planar singular graph $\gamma$  and the $sl(N)$-polynomial is then expressed as
a sum (over resolutions) of such polynomials.  Matrix factorizations
can be used to make an assignment
$$
A_N^*(-)\colon \text{Planar singular graphs} \ra \grvectQ
$$
such that $\sum_i q^i\text{dim}A_N^i(\Gamma)= \moy(\Gamma)$. The local
relations satisfied by $\moy (-)$ are lifted to $A_N^*(-)$; for
example
$$
\moy
(\hspace{0.1mm}\raisebox{-7.3mm}{\includegraphics[scale=0.17]{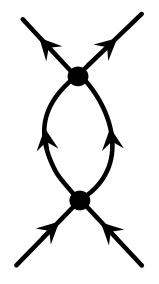}}\hspace{0.1mm}
) = (q+q^{-1})\;\;\moy (\hspace{1mm}\raisebox{-2.5mm}{\includegraphics[scale=0.1]{sres}}\hspace{1mm})
$$
becomes
$$
A_N^i
(\hspace{0.1mm}\raisebox{-7.3mm}{\includegraphics[scale=0.17]{MOYrel}}\hspace{0.1mm}
) \cong A_N^{i-1}
(\hspace{1mm}\raisebox{-2.5mm}{\includegraphics[scale=0.1]{sres}}\hspace{1mm})
\oplus
A_N^{i+1}  (\hspace{1mm}\raisebox{-2.5mm}{\includegraphics[scale=0.1]{sres}}\hspace{1mm})
$$
Moreover, if two planar singular graphs are identical except in a small neighbourhood where they are as shown above then there are maps
$$\xymatrix{A_N^*(\ores) \ar[r] &
  A_N^*(\sres)} \;\;\;\;\;\;\; \xymatrix{A_N^*(\sres) \ar[r] &
  A_N^*(\ores)}
$$ 
A functor is constructed by taking
$$
F_{D,N}(A) = A_N^*( \text{resolution associated to $A$}) 
$$
and applying the maps above. Making the complex and taking homology
defines $sl(N)$-homology. 

\begin{rema} (on gradings for which we follow the conventions given
  in Rasmussen's paper cited above). Let $A\subset X$ and let $\Gamma$ be the
  associated resolution. If $x \in A_N^k(\Gamma)$ then the
  corresponding element of the (bi-graded) complex has bi-degree
$$
(|A| -n_+, k-i+(N-1)(n_+ - n_-))
$$
  \end{rema}

\begin{rema}
   $KR_2^{**}$ should be isomorphic to $\prekh^{**}$ and indeed it
  is (Hughes, \newcite{1302.0331}).   
 \end{rema}

\begin{rema}
  There are both reduced and un-reduced versions of the theory and
  they are related by a
spectral sequence (Lewark, \newcite{1310.3100}). 
\end{rema}

Calculations with $sl(N)$-homology are much harder than for Khovanov
homology. An understanding of why this is so can be obtained by
attempting to compute the $sl(N)$-homology of Hopf link and comparing it
to the Khovanov homology calculation. Here are some calculational results.

\begin{itemize}
\item  The $sl(N)$-homology of two
  bridge knots has been completely determined and the result can be expressed in terms of the HOMFLYPT polynomial and signature
  (Rasmussen, \newciteold{0508510}).  The torus knots $T(2,n)$ are a
  special case and this computation confirms previous conjectures
  (Dunfield-Gukov-Rasmussen, \newciteold{0505662}).
\item There is an explicit conjecture about
 the  $sl(N)$-homology of 3-stranded torus knots (Gorsky and Lewark,
 \newcite{1404.0623}).
\item As $n\ra \infty$ the homology $\kr ** (T(k,n))$ stabilises in
  bounded degree (Sto\u{s}i\'c, \newciteold{0511532}) so it makes sense to consider the $sl(N)$-homology
 of $T(k,\infty)$ and there are conjectures to what this should be 
 (Gorsky, Oblomov and Rasmussen,  \newcite{1206.2226}). 
 \end{itemize}

 \begin{rema}
   One might hope that taken collectively the family of
  $sl(N)$-homologies are a complete invariant but this is not so and
  there are families of distinct knots undistinguishable by the $\kr **$ (Watson, \newciteold{0606630}; Lobb, \newcite{1105.3985}).
 \end{rema}

\begin{rema}
 There was a different construction for $N=3$ available before matrix
 factorizations entered the picture (Khovanov, \newciteold{0304375}) using Kuperberg's theory of webs. 
\end{rema}

 The construction of the local coefficient system used in the construction is once again a two step process: a geometric step is followed by an algebraic step and it is natural to ask if 
\begin{wrapfigure}{l}{3.5cm}
\vspace{-3mm}
\includegraphics[scale=0.15]{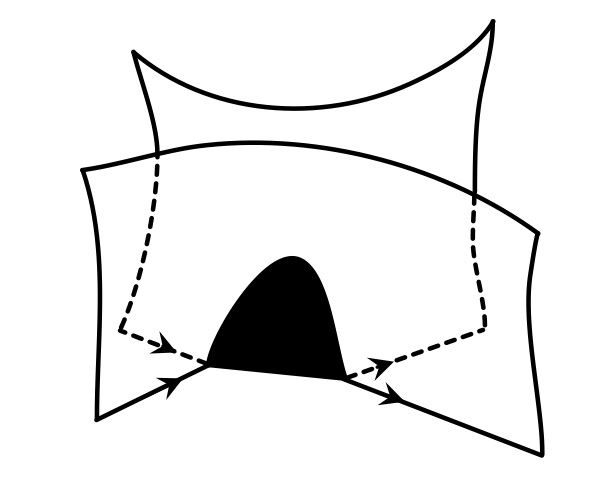}
\vspace{-7mm}
\end{wrapfigure}  
 Bar-Natan's approach can be carried over  to $sl(N)$-homology. For simplicity we have 
been
  using the language of singular link diagrams in  which we allowed crossings looking  like $\sres$.  In fact the notation used by Khovanov
  and Rozansky is to elongate the 
vertex into a {\em thick } edge (or
  {\em double  edge) }
\hspace{1mm}\raisebox{-1.9mm}{\includegraphics[scale=0.07]{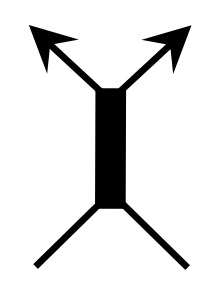}}\hspace{1mm}.
This depiction more accurately reflects the viewpoint of Murakami, Ohtsuki and Yamada and with this in  mind Bar-Natan's approach can  be
made to work for $sl(N)$-homology, but 
cobordisms must be generalised to {\em foams} 
which take
into account the existence of thick edges (Mackaay and Vaz,
\newcite{0710.0771}; Mackaay, Sto\u{s}i\'c and Vaz, \newcite{0708.2228}).

\begin{rema}
   The main new algebraic ingredient in the construction of
  $sl(N)$-homology is the notion of a matrix factorization and these
  are used locally: a matrix factorization is associated to a small
  neighbourhood of a resolved diagram. Locality of diagrams was best
  patched together (when incorporating link cobordisms too) using the
  formalism of a canopolis. There is a certain
  category of matrix factorizations that can be given the structure of a
  canopolis and the construction of $sl(N)$-homology can be presented in these terms (Webster, \newciteold{0610650}).
\end{rema}

\begin{rema}
   The construction of Khovanov homology can be modified to
  produced Lee theory and in a similar way there are ``degenerate'' variants of $sl(N)$
  homology (Gornik, \newciteold{0402266}). Like Lee theory these are filtered
  theories and can be completely computed in terms of linking
  numbers. There is also an analogue of the Lee-Rasmussen spectral
  sequence for these theories (Wu, \newciteold{0612406}). 
\end{rema}

\begin{rema}
  Throughout this section we have been assuming that we are
  working over $\bbQ$ (or $\bbC$), but integral theories have also been studied (Krasner, \newcite{0910.1790}).
\end{rema}

\subsection{Khovanov-Rozansky HOMFLYPT-homology}
For each $N$ the (normalised) $sl(N)$-polynomial  $\widetilde{P}_N $ is
a specialisation of the following version of the HOMFLYPT polynomial $\widetilde{P}$
$$
a \widetilde{P}(\posx) - a \widetilde{P} (\negx) + (q-q^{-1}) \widetilde{P} (\ores) =0
$$
and
$$
\widetilde{P} (\text{unknot}) = 1.
$$
Note that when $a=1$ this also gives the (Conway)-Alexander
polynomial. A very natural question in the context of link homology
theories is: is there a link homology theory $\barh$ whose graded
dimensions combine to give the polynomial $\widetilde{P}$ and which
specialize to $\krr **$? (In asking this question it is not
immediately clear what ``specialise'' should mean.) There is such a theory constructed once again by Khovanov and
Rozansky. 
\begin{itemize}
\item {\em Matrix factorizations and link homology II}, (Khovanov and
  Rozansky, \newciteold{0505056})
\end{itemize}
A good place to start might be the expository sections of 
\begin{itemize}
\item {\em Some differentials on Khovanov-Rozansky homology}
  (Rasmussen,  \newciteold{0607544})
\end{itemize}

Khovanov and Rozansky's {\em HOMFLYPT homology} (reduced version) assigns to each braid closure
diagram $D$ a {\em triply} graded $\bbQ$-vector space
$\barh^{*,*,*}(D)$ such that 
\begin{enumerate}
\item $D_1\sim D_2 \Longrightarrow \barh (D_1) \cong \barh (D_2)$, 
\item The HOMFLYPT polynomial is recovered
$$
\sum_{r,s,t} (-1)^{\frac{t-s}{2}}a^s q^r \text{dim}(\barh^{r,s,t}(D))
= \widetilde{P}(D),
$$
\item $\barh (\text{unknot})\cong \bbQ$ in grading $(0,0,0)$,
\item $\barh (L_1 \sqcup L_2) \cong \barh (L_1) \otimes \barh
  (L_2)\otimes \bbQ [x] $,
\item There are skein long exact sequences.
\end{enumerate}

\begin{rema}
 The notation and grading conventions we are following are
    Rasmussen's. In fact Khovanov and Rozansky work with the unreduced
    theory $H$ which is related to reduced by $H\cong \barh \otimes
    \bbQ [x]$. The reduced theory has the nice property that for any
    connected sum we have $\barh (L_1 \sharp L_2) \cong \barh (L_1) \otimes \barh (L_2)$. 
\end{rema}

\begin{rema}
   The construction of HOMFLYPT homology uses matrix
  factorizations (though see section \ref{subsec:krhoch} below). They
  are {\em graded} matrix factorizations which is the grading which
  pushes through to ultimately give three gradings. 
\end{rema}

\begin{rema}
 The theory is not as well behaved as previous theories. For one
  thing it is restricted to braid closures (a priori this could be
  removed but is required for the proof of Reidemeister
  invariance). Another drawback is that there is no functoriality.
 \end{rema}

HOMFLYPT-homology reproduces the polynomial $\widetilde{P}$ but what
about its relationship to $sl(N)$-homology? Since $\widetilde{P}$
specialises to the $sl(N)$-polynomial one would hope there is a
relationship. This question has been thoroughly investigated in a
wonderful paper which analyses a family of
spectral sequences relating these theories (Rasmussen,
\newciteold{0607544}). The main result is:

\begin{theo}
  \begin{enumerate}
  \item For each $N>0$ there exists a spectral sequence starting with
    $\barh$ and converging to $\krr **$. Moreover each page is a knot
    invariant.
\item There is a spectral sequence starting with $\barh$ and
  converging to $\bbQ$.
  \end{enumerate}
\end{theo}

As a consequence one has:
\begin{prop}
For large $N$
$$
\krr ij (L) \cong \bigoplus_{\substack{j= r+Ns\\ 2i=t-s}} \barh^{r,s,t}(L).
$$  
\end{prop}

\begin{rema}
  The existence of these spectral sequences gives information
    about the form of HOMFLYPT-homology in much the same way as the
    Lee-Rasmussen spectral sequence gives information about the form
    of Khovanov homology. As there is now a whole family of
spectral sequences to be considered this gives a large amount of
    information. 
Rasmussen carries out low crossing
    number computations, and a nice worked example of how to glean information
    from the existence of the spectral sequences is the calculation
    showing that $\barh$(Conway knot) and
    $\barh$(Kinoshita-Terasaka knot) are
    isomorphic (Mackaay and Vaz, \newcite{0812.1957}). Some of this structure  was predicted before the
    construction of the spectral sequences in particular
   for torus knots (Dunfield, Gukov, and Rasmussen, \newciteold{0505662}). 
\end{rema}

\begin{rema}
 One can construct other theories similar to
  $\krr **$ and there are Rasmussen-type spectral sequences starting
  with $\barh$ and converging to these theories (Wu,  \newciteold{0612406}).
\end{rema}

\begin{rema}
There is a spectral sequence converging to knot Floer homology
 which has $E_1$-page the HOMFLYPT-homology (Manolescu, \newcite{1108.0032} ).
  \end{rema}

\section{Generalisations: further developments}

\subsection{Other constructions of Khovanov-Rozansky HOMFLYPT-homology}\label{subsec:krhoch}

There is an alternative construction of triply-graded HOMFLYPT-homology which uses Hochschild homology (Khovanov, \newciteold{0510265}). 
One can associate a cochain complex $F^*(\sigma)$ to a word $\sigma$ representing a braid group element (Rouquier, \newciteold{0409593}). This assignment is  such that if two words represent the same group element then the associated complexes are isomorphic. Khovanov uses this construction in the following way. Suppose we have a link presented as the closure of an $m$-braid diagram $D$ and let $\sigma$ be the corresponding braid word and $F^*(\sigma)$ its Rouquier complex. Now apply Hochschild homology $H\!H(R, -)$ to this to get a complex
$$
\xymatrix{
\cdots \ar[r] & H\!H(R,F^i(\sigma)) \ar[r]& H\!H(R,F^{i+1}(\sigma) ) \ar[r]& H\!H(R,F^{i+2}(\sigma)) \ar[r] & \cdots
}
$$
where $R$ is a certain ring ($R=\bbQ[x_1-x_2, \ldots , x_{m-1}-x_m]\subset \bbQ[x_1, \ldots , x_m]$). There are internal gradings and each term in the sequence is in fact bigraded.

\begin{theo}
  The homology of this complex, denoted ${\mathcal H}^{*,*,*}$, is independent (up to isomorphism) of the choices made and (module juggling grading conventions) is isomorphic to Khovanov and Rozansky's HOMFLYPT-homology. 
\end{theo}

There is another more geometric construction of HOMFLYPT homology which uses the cohomology of sheaves on certain algebraic groups  (Webster and Williamson, \newcite{0905.0486}).

\subsection{Coloured link homologies}
Coloured Jones polynomials arise from $sl(2)$ using higher
dimensional representations rather than the fundamental two
dimensional one. The natural question here is to ask if there are link homology theories that stand in the same relationship
to these polynomials as Khovanov homology does to the Jones polynomial. There is a formula for each coloured Jones polynomial as a sum of Jones polynomials of cables and this can be also be used to define link homology theories.

\begin{itemize}
\item {\em Categorifications of the colored Jones polynomial} (Khovanov, \newciteold{0302060})
\end{itemize}
For this to work coefficients in $\bbZ/2$ are needed, but can be extended to $\bbZ[\frac{1}{2}]$.
\begin{itemize}
\item {\em Categorification of the colored Jones polynomial and Rasmussen invariant of links} (Beliakova-Wehrli, \newciteold{0510382})
\end{itemize}
The $sl(N)$-polynomial is associated to the fundamental representation of $sl(N)$ and by allowing other representations one obtains coloured versions. Khovanov and Rozansky's $sl(N)$- homology can be extended in a similar way.
\begin{itemize}
\item {\em Generic deformations of the colored {${sl}(N)$}-homology} (Wu, \newcite{0907.0695})
\end{itemize}
Using the definition of HOMFLYPT homology in terms of Hochschild homology it is possible to consider coloured versions of this theory too.
\begin{itemize}
\item {\em The 1,2-colored HOMFLY-PT link homology} (Mackaay, Sto\u{s}i\'c and Vaz, \newcite{0809.0193})
\end{itemize}
The algebro-geometric construction of HOMFLYPT homology can be extended to a coloured version which agrees with the above when restricted.
\begin{itemize}
\item {\em A geometric construction of colored HOMFLYPT homology} (Webster and Williamson, \newcite{0905.0486})
\end{itemize}

\subsection{Higher representation theory}
This is now a vast and important subject providing the most
comprehensive answers to the question of how one should generalise
Khovanov homology to other Lie algebras. To get an idea of the state of play you could read the introductory sections of:

\begin{itemize}
\item {\em Khovanov homology is a skew Howe 2-representation of categorified quantum $sl_m$} (Lauda, Queffelec and Rose, \newcite{1212.6076})
\end{itemize}
\begin{itemize}
\item {\em An introduction to diagrammatic algebra and categorified quantum $sl_2$} (Lauda, \newcite{1106.2128})
\end{itemize}
\begin{itemize}
\item {\em Knot invariants and higher representation theory } (Webster, \newcite{1309.3796})
\end{itemize}
For the latter, the theory is explained separately in detail for the $sl(2)$ case (Webster, \newcite{1312.7357}).

The work of Rouquier has been of fundamental importance in this area and you can get an idea of his vision from:
\begin{itemize}
\item {\em Quiver {H}ecke algebras and 2-{L}ie algebras} (Rouquier, \newcite{0812.5023})
\end{itemize}

\section{Applications of Khovanov homology}
\subsection{Concordance invariants}

The first paper to read on this subject is 
\begin{itemize}
\item {\em Khovanov homology and the slice genus} (Rasmussen, \newciteold{0402131})
\end{itemize}
Recall that for a knot the Lee-Rasmussen spectral sequence leaves only two generators on the $E_\infty$-page. If we use the grading conventions which impose the differentials on the usual picture for Khovanov homology, then denoting the $E_\infty$ page by $\bbK_\infty^{*,*}$, the statement that the spectral sequence {\em converges} to Lee theory means that 
$$
\bbK_\infty^{i,j} = \frac{F^j \lee i}{ F^{j+1} \lee i}
$$
where $F^* \lee i$ is the induced filtration on Lee theory. Since we
are working over $\bbQ$ (so the spectral sequence has no extension problems) this means that  $\lee i \cong \bigoplus_j \bbK_\infty^{i,j}$. 

A priori the filtration grading (here the grading denoted by $j$) of
the $E_\infty$-page of a spectral sequence is not particularly
meaningful, but in this case the entire spectral sequence from the
second page onwards is a knot invariant and thus the filtration
gradings of the generators (two of them) surviving to the
$E_\infty$-page are too. In fact these two generators lie in
filtration gradings that differ by two.  

\begin{prop}
 For a knot $K$ there exists an even integer $s(K)$ such that the two
 surviving generators in the Lee-Rasmussen spectral sequence have
 filtration degrees $s(K) \pm 1$.
\end{prop}

\begin{defi}
  The integer $s(K)$ is called the {\em Rasmussen  $s$-invariant} of the knot $K$.
\end{defi}     

\begin{rema}
For an alternating knot, Rasmussen's invariant agrees with the signature.  
\end{rema}

By digging down a bit into the filtration, Rasmussen shows that his invariant has the following properties:
\begin{enumerate}
\item $s(\text{unknot}) = 0$,
\item $s(K_1 \sharp K_2) = s(K_1) + s(K_2)$,
\item $s(K^!) = - s(K)$.
\end{enumerate}
A cobordism $\Sigma\colon K_1 \ra K_2$ induces a filtered map
$Lee(\Sigma)\colon \lee * (K_1) \ra \lee * (K_2)$ of filtered degree
$\chi(\Sigma)$ meaning that $\im(F^j\lee * (K_1) \subset
F^{j+\chi(\Sigma)}\lee * (K_2)$. Denoting the filtration grading by
$gr$ this means that for $\alpha\in \lee * (K)$ we have
$$
gr(Lee(\Sigma)(\alpha) \geq gr(\alpha) + \chi(\Sigma)
$$
Also, Rasmussen shows:

\begin{prop}
 If $\Sigma$ is connected then $Lee (\Sigma)$ is an isomorphism. 
\end{prop}

Using these properties one can show that Rasmussen's invariant provides an obstruction to a knot being smoothly
slice. (Recall that all manifolds and link cobordims are assumed to be smooth).

\begin{prop}\label{prop:slice}
  If $K$ is a smoothly slice knot then $s(K)=0$
\end{prop}

\begin{proof}
  Let $\Sigma$ be a slice disc with another small disc removed. This
  can be viewed as (connected) link cobordism $\Sigma\colon K \ra U$
  (the unknot). Since the Euler characteristic of $\Sigma$ is zero,
  this cobordism induces a filtered isomorphism of filtered degree
  zero 
$$Lee(\Sigma)\colon \lee 0 (K) \ra \lee 0 (U).$$
 Thus for any
  $\alpha\in \lee 0 (K)$ we have 
$$gr(Lee(\Sigma)(\alpha)) \geq
  gr(\alpha).$$
Now $\lee 0 (U)$ has two generators in filtration
  degrees $\pm 1$ and  $Lee(\Sigma)$ is an isomorphism, from which we have
  $-1 \leq gr(Lee(\Sigma)(\alpha))\leq 1$ giving
$$gr(\alpha)\leq gr(Lee(\Sigma)(\alpha))\leq 1.$$
 Now $s(K)$ is equal to $gr(\alpha)-1$ for {\em some} $\alpha$ so
  $s(K)\leq 0$. Finally, a similar argument applies to $K^!$ giving
  $s(K^!)\leq 0 $ and so $s(K)= -s(K^!)\geq 0$.
\end{proof}

\begin{rema}
This proof uses the fact that Lee theory is a {\em functor}.  
\end{rema}

In fact more is true and $s$ gives a lower bound for the slice genus.

\begin{theo}
 Rasmussen's invariant is a concordance invariant and for a knot $K$ 
 $$|s(K)| \leq 2 g_s(K),
$$
where $g_s(K)$ denotes the smooth slice genus of $K$.
\end{theo}

\begin{exer}
Prove this theorem by modifying the proof of Proposition
\ref{prop:slice} above. 
\end{exer}

\begin{rema}
  By studying the $s$-invariant for positive knots and using the
  theorem above, Rasmussen gives a simple
  proof of the Milnor conjecture (Rasmussen, \newciteold{0402131}):
  the slice genus of the torus knot $T(p,q)$ is
  $\frac{1}{2}(p-1)(q-1)$. 
\end{rema}

\begin{rema}
  Gompf gives a way of constructing non-standard smooth structures on
  $\bbR^4$ from the data of a topologically slice but not smoothly
  slice knot. By work of Freedman if the Alexander polynomial
  $\Delta_K$ is 1 then $K$ is topologically slice. Thus a non-standard
  smooth structure on $\bbR^4$ can be inferred from a knot $K$
  satisfying $\Delta_K=1$ and $s(K) \neq 0$. Examples of such knots are
  readily found, for example, the pretzel knot $P(-3,5,7)$.
\end{rema}

\begin{rema}
 There is a similar invariant to Rasmussen's, called the $\tau$-invariant, coming from
Heegaard-Floer knot homology. While in many cases $2\tau=s$  in
general this is not the case (Hedden and Ording,
\newciteold{0512348}). There is even an example of a topologically
slice knot for which $s\neq 2\tau$ (Livingston, \newciteold{0602631}).
\end{rema}

\begin{rema}
  For a short time it looked like Rasmussen's invariant might help
  to find a counter-example to the smooth 4-dimensional Poincar\'e
  conjecture (Freedman, Gompf, Morrison and Walker, \newcite{0906.5177}),
  but this hope was short lived and the potential counter-examples are
  all standard spheres (Akbulut, \newcite{0907.0136}).  In fact
 Rasmussen's invariant can be related to a
  similar invariant from instanton homology which leads to the
  conclusion that Rasmussen's invariant will never detect
  counter-examples  to the 4-dimensional Poincar\'e conjecture  of this nature
 (Kronheimer and Mrowka, \newcite{1110.1297}).
\end{rema}

By replacing Khovanov homology by $sl(N)$-homology and allowing
Gornik's theory $G_N^*$  to play the role of Lee theory, there is once again a
spectral sequence starting with the former and converging to the
latter. Since Gorniks theory has dimension $N$ concentrated in (homological) degree
0 one can get a Rasmussen-like invariant (Wu, \newciteold{0612406};
Lobb, \newcite{0702393}; Lobb, \newcite{1012.2802}). 

\begin{theo}
  (1) Let $K$ be a knot. There exists and integer $s_N(K)$ such that
$$
\sum q^j G_N^{0,j}(K) = q^{s_N(K)}\frac{q^N - q^{-N}}{q-q^{-1}}
$$
where the second grading on Gornik theory is the filtration grading.\\
(2) This provides a lower bound for the slice genus:
$$|s_N(K)| \leq 2(n-1)g_s(K).$$
\end{theo}

\begin{rema}
 It is interesting to ask if these invariants are related or not for
 various $N$.  Lewark conjectures that the invariants
 $\{s_N(K)\}_{N\geq 2}$ are linearly independent with evidence from
 the result that  $s_2(K)$ is not a linear combination of
 $\{s_N(K)\}_{N\geq 3}$ and a similar statement for $s_3(K)$ (Lewark, \newcite{1310.3100}). 
\end{rema}

Another way of obtaining a
Rasmussen-type invariant is by using the spectral sequence to
Bar-Natan theory. This gives invariants $s_R^{BN}(K)$ for a variety
of rings $R$. It was thought (incorrectly) that over $\bbZ$, $\bbQ$
and finite fields that these invariants always coincide with
Rasmussen's original invariant (Mackaay, Turner and Vaz, \newciteold{0509692}) but
this is wrong and Cotton Seed has done some calculations which show
that the knot $K=K14n19265$ has  $s(K)\neq s^{BN}_{\bbF_2}(K)$. These
invariants have been further refined (Lipshitz and Sarkar,
\newcite{1206.3532}) and a discussion of the $K14n19265$ example can be found
there.

\begin{rema}
  For links (rather than knots) there is also a way to obtain a Rasmussen-type invariant (Beliakova and
  Wehrli, \newciteold{0510382}). 
\end{rema}

\subsection{Unknot detection}
Khovanov homology is a nice functorial invariant which is known not to
be {\em complete} and it is not hard to find distinct knots with the
same Khovanov homology. However, the weaker question of whether or not
Khovanov homology detects the unknot remained open until recently. 

There are a
number of partial results applying somewhat the same approach: make something
else out of the knot and use a spectral sequence to Heegaard-Floer
homology to make a conclusion about the minium size of the $E_2$-page 
(Hedden and Watson, \newcite{0805.4423}; Hedden, \newcite{0805.4418}; Grigsby and
Wehrli, \newcite{0807.1432}). 
For example, the following is a result of Hedden and Watson:
\begin{theo}
 The dimension of the reduced Khovanov homology of the (2,1)-cable of a knot $K$ is exactly 1 if and only if $K$ is the unknot.
\end{theo}

It is now known that Khovanov homology itself detects the unknot
(Kronheimer and Mrowka, \newcite{1005.4346}).

\begin{theo}
  Khovanov homology detects the unknot.
\end{theo}

This result also uses a spectral sequence but this time using another
theory defined using instantons (Kronheimer and Mrowka,
\newcite{0806.1053}). This is a deep result requiring mastery of huge
amount of low dimensional topology.

\subsection{Other applications}

After the slice genus one of the first places Khovanov homology found application was to bounding the Thurston-Bennequin number  (Shumakovitch, \newciteold{0411643}; Plamenevskaya, \newciteold{0412184}; Ng, \newciteold{0508649}).  Let $L$ be a link and let $\overline{tb}(L)$ be the maximum Thurston-Bennequin number over all Legendrian representatives of $L$. The result of Ng is:
\begin{theo} There is a bound for maximum Thurston-Bennequin number give by
$$
 \overline{tb}(L) \leq \text{min} \{ k \mid \oplus_{j-i=k} \kh i j L \neq 0\} 
$$
and this bound is sharp for alternating links. 
\end{theo}

The reduced Khovanov homology of an alternating knot lies exclusively on the line $j-2i = \text{signature(K)}$ and as a general rule the Khovanov homology of a link clusters around this line. The  {\em homological   width} of a link  the width of the diagonal band in which the non-trivial Khovanov homology lies: if
$$
w_{max}=\max\{j-2i\mid \kh i j L \neq 0\} \;\;\;\;\; w_{mim}=\min\{j-2i\mid \kh i j L \neq 0\} 
$$ 
then $\text{width}(L) = w_{max}-w_{min} + 1$. Alternating links, for example,  have width 1. Since width involves both gradings available to Khovanov homology it is revealing something new not available to, say, the Jones polynomial. It can be used to provide certain obstructions to Dehn fillings (Watson, \newcite{0807.1341}; Watson, \newcite{{1010.3051}}).

In a somewhat different direction, there is also (vector space-valued) invariant of tangles using an natural inverse system of Khovanov homology groups which can be applied  to strongly invertible knots to show that  a strongly invertible knot is the trivial knot if and only if the invariant is trivial (Watson, \newcite{1311.1085}).

It is possible to use a variant of Khovanov homology to distinguish
between braids and other tangles (Grigsby and Ni, \newcite{1305.2183}).

\section{Geometrical interpretations and related theories}
What is Khovanov homology {\em really}? What geometrical features of knots and links does it measure? Is there an intrinsic definition starting with an actual knot in $S^3$ rather than a diagrammatic representation of it? The Jones polynomial has a good ``physical'' interpretation - how about Khovanov homology? Many of the ingredients of Khovanov homology are familiar to other areas of mathematics - what bridges can be built? In this brief final section we gather together a few places where these questions have been addressed.

\subsection{Symplectic  geometry}
There is a way to define a (singly) graded vector space invariant of
links by using Lagrangian Floer cohomology. This approach has the very
nice property that it starts with
an actual link rather than a diagram. It is conjectured to be
isomorphic to Khovanov homology after collapsing the grading of the later.
\begin{itemize}
\item {\em A link invariant from the symplectic geometry of nilpotent slices} (Seidel and Smith, \newciteold{0405089})
\end{itemize}
This approach has been generalised to Khovanov-Rozansky $sl(N)$-homologies:
\begin{itemize}
\item {\em Link homology theories from symplectic geometry} (Manolescu, \newciteold{0601629})
\end{itemize}

\subsection{Knot groups and representation varieties}
For low crossing number
examples the (singly graded) Khovanov homology of a link is isomorphic
to a graded group constructed from the cohomology of the space of
$\mathit{SU}(2)$ representations of the fundamental group of the link
complement. These latter spaces can be understood in terms of
intersections of Lagrangian submanifolds of a certain symplectic manifold.

\begin{itemize}
\item {\em Symplectic topology of SU(2)-representation varieties and link homology, I: symplectic braid action and the first Chern class} (Jacobsson and Rubinsztein, \newcite{0806.2902})
\end{itemize}

\subsection{Instanton knot homology}
The observation connecting Khovanov homology to $\mathit{SU}(2)$-representation
varieties is also the starting point for the construction of a
functorial link homology theory defined as an instanton Floer homology
theory. This is the theory used to show that Khovanov homology detects
the unknot. 

\begin{itemize}
\item {\em Knot homology groups from instantons} (Kronheimer and Mrowka,\newcite{0806.1053})
\end{itemize}
\begin{itemize}
\item {\em Filtrations on instanton homology} (Kronheimer and Mrowka, \newcite{1110.1290})
\end{itemize}
\begin{itemize}
\item {\em Gauge theory and Rasmussen's invariant} (Kronheimer and Mrowka, \newcite{1110.1297})
\end{itemize}

\subsection{Derived categories of coherent sheaves}
There is an algebro-geometric construction of a link homology theory
isomorphic to  Khovanov homology.
\begin{itemize}
\item {\em Knot homology via derived categories of coherent sheaves I, sl(2) case} (Cautis and Kamnitzer, \newciteold{0701194})
\end{itemize}

\subsection{Physics}
The Jones polynomial (famously) has a description as a
path integral of a 3-dimensional gauge theory using the Chern-Simons
action.  Mathematical physics and string theory also have things to
say about Khovanov homology.

\begin{itemize}
\item {\em Khovanov-Rozansky homology and topological strings } (Gukov, Schwarz and Vafa, \cite{0412243}) {\mbox{}\marginpar{\hspace{0pt}\small{{ \href{http://arxiv.org/abs/hep-th/0412243}{0412243}}}}}
\end{itemize}

\begin{itemize}
\item {\em Link homologies and the refined topological vertex} (Gukov, Iqbal, Kozcaz and Vafa, \newcite{0705.1368})
\end{itemize}

\begin{itemize}
\item {\em Fivebranes and knots} (Witten, \newcite{1101.3216})
\end{itemize}

\begin{itemize}
\item {\em Khovanov homology and gauge theory}  (Witten, \newcite{1108.3103} )
\end{itemize}

\begin{itemize}
\item {\em Two lectures on the Jones polynomial and Khovanov homology}  (Witten, \newcite{1401.6996} )
\end{itemize}

\subsection{Homotopy theory}
Homotopy theory is a subject rich in computational and theoretical tools and it would be nice to have a way of applying these to Khovanov homology. One basic question to ask is if the Khovanov homology of a link can be obtained from the application of some classical invariant from algebraic topology (cohomology for example) to a space defined from the link. This is the notion of Khovanov homotopy type.
\begin{itemize}
\item {\em A Khovanov homotopy type} (Lipshitz and Sarkar, \newcite{1112.3932})
\end{itemize}

One can also carry out the entire construction of Khovanov homology in a homotopy theoretic setting using Eilenberg-Mac Lane spaces, homotopy limits and the description of Khovanov homology in terms of right derived functors of the inverse limit. 
\begin{itemize}
\item {\em The homotopy theory of  Khovanov homology} (Everitt and Turner, \newcite{1112.3460})
\end{itemize}

\subsection{Factorization homology}
A vast machine from algebraic topology (factorization homology of
singular manifolds) may -  on specialising - provide new link homology
theories related to Khovanov homology. 

\begin{itemize}
\item {\em Structured singular manifolds and factorization homology} (Ayala, Francis and Tanaka, \newcite{1206.5164})
\end{itemize}

\bibliographystyle{halpha.bst}
\bibliography{linkhomology}

\begin{thebibliography}{FGMW10}

\bibitem[AFT12]{1206.5164}
David Ayala, John Francis, and Hiro~Lee Tanaka.
\newblock Structured singular manifolds and factorization homology.
\newblock 2012, 1206.5164.

\bibitem[Akb10]{0907.0136}
Selman Akbulut.
\newblock Cappell-{S}haneson homotopy spheres are standard.
\newblock {\em Ann. of Math. (2)}, 171(3):2171--2175, 2010.

\bibitem[APS04]{0410238}
Marta~M. Asaeda, J{\'o}zef~H. Przytycki, and Adam~S. Sikora.
\newblock Categorification of the {K}auffman bracket skein module of
  {$I$}-bundles over surfaces.
\newblock {\em Algebr. Geom. Topol.}, 4:1177--1210 (electronic), 2004.

\bibitem[Bei12]{1205.2256}
Simon Beier.
\newblock An integral lift, starting in odd {K}hovanov homology, of
  {S}zab{\'o}'s spectral sequence.
\newblock 2012, 1205.2256.

\bibitem[Bla10]{1405.7246}
Christian Blanchet.
\newblock An oriented model for {K}hovanov homology.
\newblock {\em J. Knot Theory Ramifications}, 19(2):291--312, 2010.

\bibitem[Blo10]{0903.3746}
Jonathan~M. Bloom.
\newblock Odd {K}hovanov homology is mutation invariant.
\newblock {\em Math. Res. Lett.}, 17(1):1--10, 2010.

\bibitem[BN02]{0201043}
Dror Bar-Natan.
\newblock On {K}hovanov's categorification of the {J}ones polynomial.
\newblock {\em Algebr. Geom. Topol.}, 2:337--370 (electronic), 2002.

\bibitem[BN05]{0410495}
Dror Bar-Natan.
\newblock Khovanov's homology for tangles and cobordisms.
\newblock {\em Geom. Topol.}, 9:1443--1499, 2005.

\bibitem[BN07]{0606318}
Dror Bar-Natan.
\newblock Fast {K}hovanov homology computations.
\newblock {\em J. Knot Theory Ramifications}, 16(3):243--255, 2007.

\bibitem[BNM06]{0606542}
Dror Bar-Natan and Scott Morrison.
\newblock The {K}aroubi envelope and {L}ee's degeneration of {K}hovanov
  homology.
\newblock {\em Algebr. Geom. Topol.}, 6:1459--1469, 2006.

\bibitem[BW08]{0510382}
Anna Beliakova and Stephan Wehrli.
\newblock Categorification of the colored {J}ones polynomial and {R}asmussen
  invariant of links.
\newblock {\em Canad. J. Math.}, 60(6):1240--1266, 2008.

\bibitem[BW10]{0910.5050}
Anna Beliakova and Emmanuel Wagner.
\newblock On link homology theories from extended cobordisms.
\newblock {\em Quantum Topol.}, 1(4):379--398, 2010.

\bibitem[Cap08]{0707.3051}
Carmen~Livia Caprau.
\newblock {$\rm sl(2)$} tangle homology with a parameter and singular
  cobordisms.
\newblock {\em Algebr. Geom. Topol.}, 8(2):729--756, 2008.

\bibitem[CK08]{0701194}
Sabin Cautis and Joel Kamnitzer.
\newblock Knot homology via derived categories of coherent sheaves. {I}. {T}he
  {${ {sl}}(2)$}-case.
\newblock {\em Duke Math. J.}, 142(3):511--588, 2008.

\bibitem[CMW09]{0701339}
David Clark, Scott Morrison, and Kevin Walker.
\newblock Fixing the functoriality of {K}hovanov homology.
\newblock {\em Geom. Topol.}, 13(3):1499--1582, 2009.

\bibitem[DGR06]{0505662}
Nathan~M. Dunfield, Sergei Gukov, and Jacob Rasmussen.
\newblock The superpolynomial for knot homologies.
\newblock {\em Experiment. Math.}, 15(2):129--159, 2006.

\bibitem[ET12]{1112.3460}
Brent Everitt and Paul Turner.
\newblock The homotopy theory of {K}hovanov homology.
\newblock 2012, 1112.3460.

\bibitem[FGMW10]{0906.5177}
Michael Freedman, Robert Gompf, Scott Morrison, and Kevin Walker.
\newblock Man and machine thinking about the smooth 4-dimensional {P}oincar\'e
  conjecture.
\newblock {\em Quantum Topol.}, 1(2):171--208, 2010.

\bibitem[GIKV10]{0705.1368}
Sergei Gukov, Amer Iqbal, Can Koz{\c{c}}az, and Cumrun Vafa.
\newblock Link homologies and the refined topological vertex.
\newblock {\em Comm. Math. Phys.}, 298(3):757--785, 2010.

\bibitem[GL14]{1404.0623}
Eugene Gorsky and Lukas Lewark.
\newblock On stable sl3-homology of torus knots.
\newblock 2014, 1404.0623.

\bibitem[GN13]{1305.2183}
Elisenda Grigsby and Yi~Ni.
\newblock Sutured {K}hovanov homology distinguishes braids from other tangles.
\newblock 2013, 1305.2183.

\bibitem[Gor04]{0402266}
Bojan Gornik.
\newblock Note on {K}hovanov link cohomology.
\newblock 2004, 0402266.

\bibitem[GOR13]{1206.2226}
Eugene Gorsky, Alexei Oblomkov, and Jacob Rasmussen.
\newblock On stable {K}hovanov homology of torus knots.
\newblock {\em Exp. Math.}, 22(3):265--281, 2013.

\bibitem[GSV05]{0412243}
Sergei Gukov, Albert Schwarz, and Cumrun Vafa.
\newblock Khovanov-{R}ozansky homology and topological strings.
\newblock {\em Lett. Math. Phys.}, 74(1):53--74, 2005.

\bibitem[GW10]{0807.1432}
J.~Elisenda Grigsby and Stephan~M. Wehrli.
\newblock On the colored {J}ones polynomial, sutured {F}loer homology, and knot
  {F}loer homology.
\newblock {\em Adv. Math.}, 223(6):2114--2165, 2010.

\bibitem[Hed09]{0805.4418}
Matthew Hedden.
\newblock Khovanov homology of the 2-cable detects the unknot.
\newblock {\em Math. Res. Lett.}, 16(6):991--994, 2009.

\bibitem[HO08]{0512348}
Matthew Hedden and Philip Ording.
\newblock The {O}zsv\'ath-{S}zab\'o and {R}asmussen concordance invariants are
  not equal.
\newblock {\em Amer. J. Math.}, 130(2):441--453, 2008.

\bibitem[Hug13]{1302.0331}
Mark Hughes.
\newblock A note on khovanov-rozansky {$sl_2$}-homology and ordinary khovanov
  homology.
\newblock 2013, 1302.0331.

\bibitem[HW10]{0805.4423}
Matthew Hedden and Liam Watson.
\newblock Does {K}hovanov homology detect the unknot?
\newblock {\em Amer. J. Math.}, 132(5):1339--1345, 2010.

\bibitem[Jac04]{0206303}
Magnus Jacobsson.
\newblock An invariant of link cobordisms from {K}hovanov homology.
\newblock {\em Algebr. Geom. Topol.}, 4:1211--1251 (electronic), 2004.

\bibitem[Jac08]{0806.2902}
Magnus Jacobsson.
\newblock Symplectic topology of su(2)-representation varieties and link
  homology, i: {S}ymplectic braid action and the first {C}hern class.
\newblock 2008, 0806.2902.

\bibitem[Kho00]{9908171}
Mikhail Khovanov.
\newblock A categorification of the {J}ones polynomial.
\newblock {\em Duke Math. J.}, 101(3):359--426, 2000.

\bibitem[Kho02]{0103190}
Mikhail Khovanov.
\newblock A functor-valued invariant of tangles.
\newblock {\em Algebr. Geom. Topol.}, 2:665--741, 2002.

\bibitem[Kho04]{0304375}
Mikhail Khovanov.
\newblock sl(3) link homology.
\newblock {\em Algebr. Geom. Topol.}, 4:1045--1081, 2004.

\bibitem[Kho05]{0302060}
Mikhail Khovanov.
\newblock Categorifications of the colored {J}ones polynomial.
\newblock {\em J. Knot Theory Ramifications}, 14(1):111--130, 2005.

\bibitem[Kho06a]{0207264}
Mikhail Khovanov.
\newblock An invariant of tangle cobordisms.
\newblock {\em Trans. Amer. Math. Soc.}, 358(1):315--327, 2006.

\bibitem[Kho06b]{0411447}
Mikhail Khovanov.
\newblock Link homology and {F}robenius extensions.
\newblock {\em Fund. Math.}, 190:179--190, 2006.

\bibitem[Kho07]{0510265}
Mikhail Khovanov.
\newblock Triply-graded link homology and {H}ochschild homology of {S}oergel
  bimodules.
\newblock {\em Internat. J. Math.}, 18(8):869--885, 2007.

\bibitem[KM11a]{1005.4346}
P.~B. Kronheimer and T.~S. Mrowka.
\newblock Khovanov homology is an unknot-detector.
\newblock {\em Publ. Math. Inst. Hautes \'Etudes Sci.}, (113):97--208, 2011.

\bibitem[KM11b]{0806.1053}
P.~B. Kronheimer and T.~S. Mrowka.
\newblock Knot homology groups from instantons.
\newblock {\em J. Topol.}, 4(4):835--918, 2011.

\bibitem[KM13]{1110.1297}
P.~B. Kronheimer and T.~S. Mrowka.
\newblock Gauge theory and {R}asmussen's invariant.
\newblock {\em J. Topol.}, 6(3):659--674, 2013.

\bibitem[KM14]{1110.1290}
Peter Kronheimer and Tomasz Mrowka.
\newblock Filtrations on instanton homology.
\newblock {\em Quantum Topol.}, 5(1):61--97, 2014.

\bibitem[KR08a]{0401268}
Mikhail Khovanov and Lev Rozansky.
\newblock Matrix factorizations and link homology.
\newblock {\em Fund. Math.}, 199(1):1--91, 2008.

\bibitem[KR08b]{0505056}
Mikhail Khovanov and Lev Rozansky.
\newblock Matrix factorizations and link homology. {II}.
\newblock {\em Geom. Topol.}, 12(3):1387--1425, 2008.

\bibitem[Kra10]{0910.1790}
Daniel Krasner.
\newblock Integral {HOMFLY}-{PT} and {${\rm sl}(n)$}-link homology.
\newblock {\em Int. J. Math. Math. Sci.}, pages Art. ID 896879, 25, 2010.

\bibitem[Lau12]{1106.2128}
Aaron~D. Lauda.
\newblock An introduction to diagrammatic algebra and categorified quantum {$
  {sl}_2$}.
\newblock {\em Bull. Inst. Math. Acad. Sin. (N.S.)}, 7(2):165--270, 2012.

\bibitem[Lee05]{0201105}
Eun~Soo Lee.
\newblock An endomorphism of the {K}hovanov invariant.
\newblock {\em Adv. Math.}, 197(2):554--586, 2005.

\bibitem[Lew14]{1310.3100}
Lukas Lewark.
\newblock Rasmussen's spectral sequences and the {$ {sl}_N$}-concordance
  invariants.
\newblock {\em Adv. Math.}, 260:59--83, 2014.

\bibitem[Liv08]{0602631}
Charles Livingston.
\newblock Slice knots with distinct {O}zsv\'ath-{S}zab\'o and {R}asmussen
  invariants.
\newblock {\em Proc. Amer. Math. Soc.}, 136(1):347--349 (electronic), 2008.

\bibitem[Lob09]{0702393}
Andrew Lobb.
\newblock A slice genus lower bound from {${\rm sl}(n)$} {K}hovanov-{R}ozansky
  homology.
\newblock {\em Adv. Math.}, 222(4):1220--1276, 2009.

\bibitem[Lob12]{1012.2802}
Andrew Lobb.
\newblock A note on {G}ornik's perturbation of {K}hovanov-{R}ozansky homology.
\newblock {\em Algebr. Geom. Topol.}, 12(1):293--305, 2012.

\bibitem[Lob14]{1105.3985}
Andrew Lobb.
\newblock The {K}anenobu knots and {K}hovanov-{R}ozansky homology.
\newblock {\em Proc. Amer. Math. Soc.}, 142(4):1447--1455, 2014.

\bibitem[LP09]{0606331}
Aaron~D. Lauda and Hendryk Pfeiffer.
\newblock Open-closed {TQFTS} extend {K}hovanov homology from links to tangles.
\newblock {\em J. Knot Theory Ramifications}, 18(1):87--150, 2009.

\bibitem[LQR12]{1212.6076}
Aaron Lauda, Hoel Queffelec, and David Rose.
\newblock Khovanov homology is a skew {H}owe 2-representation of categorified
  quantum {$ {sl}(m)$}.
\newblock 2012, 1212.6076.

\bibitem[LS14a]{1112.3932}
Robert Lipshitz and Sucharit Sarkar.
\newblock A {K}hovanov stable homotopy type.
\newblock {\em J. Amer. Math. Soc.}, 27(4):983--1042, 2014.

\bibitem[LS14b]{1206.3532}
Robert Lipshitz and Sucharit Sarkar.
\newblock A refinement of {R}asmussen's {$S$}-invariant.
\newblock {\em Duke Math. J.}, 163(5):923--952, 2014.

\bibitem[Man07]{0601629}
Ciprian Manolescu.
\newblock Link homology theories from symplectic geometry.
\newblock {\em Adv. Math.}, 211(1):363--416, 2007.

\bibitem[Man14]{1108.0032}
Ciprian Manolescu.
\newblock An untwisted cube of resolutions for knot {F}loer homology.
\newblock {\em Quantum Topol.}, 5(2):185--223, 2014.

\bibitem[MSV09]{0708.2228}
Marco Mackaay, Marko Sto{\v{s}}i{\'c}, and Pedro Vaz.
\newblock {$ {sl}(N)$}-link homology {$(N\geq 4)$} using foams and the
  {K}apustin-{L}i formula.
\newblock {\em Geom. Topol.}, 13(2):1075--1128, 2009.

\bibitem[MSV11]{0809.0193}
Marco Mackaay, Marko Sto{\v{s}}i{\'c}, and Pedro Vaz.
\newblock The {$1,2$}-coloured {HOMFLY}-{PT} link homology.
\newblock {\em Trans. Amer. Math. Soc.}, 363(4):2091--2124, 2011.

\bibitem[MTV07]{0509692}
Marco Mackaay, Paul Turner, and Pedro Vaz.
\newblock A remark on {R}asmussen's invariant of knots.
\newblock {\em J. Knot Theory Ramifications}, 16(3):333--344, 2007.

\bibitem[MV08a]{0710.0771}
Marco Mackaay and Pedro Vaz.
\newblock The foam and the matrix factorization {$\rm sl_3$} link homologies
  are equivalent.
\newblock {\em Algebr. Geom. Topol.}, 8(1):309--342, 2008.

\bibitem[MV08b]{0812.1957}
Marco Mackaay and Pedro Vaz.
\newblock The reduced homfly-pt homology for the {C}onway and the
  {K}inoshita-{T}erasaka knots.
\newblock 2008, 0812.1957.

\bibitem[Nao06]{0603347}
Gad Naot.
\newblock The universal {K}hovanov link homology theory.
\newblock {\em Algebr. Geom. Topol.}, 6:1863--1892 (electronic), 2006.

\bibitem[Ng05]{0508649}
Lenhard Ng.
\newblock A {L}egendrian {T}hurston-{B}ennequin bound from {K}hovanov homology.
\newblock {\em Algebr. Geom. Topol.}, 5:1637--1653, 2005.

\bibitem[ORS13]{0710.4300}
Peter~S. Ozsv{\'a}th, Jacob Rasmussen, and Zolt{\'a}n Szab{\'o}.
\newblock Odd {K}hovanov homology.
\newblock {\em Algebr. Geom. Topol.}, 13(3):1465--1488, 2013.

\bibitem[OS05]{0309170}
Peter Ozsv{\'a}th and Zolt{\'a}n Szab{\'o}.
\newblock On the {H}eegaard {F}loer homology of branched double-covers.
\newblock {\em Adv. Math.}, 194(1):1--33, 2005.

\bibitem[Pla06]{0412184}
Olga Plamenevskaya.
\newblock Transverse knots and {K}hovanov homology.
\newblock {\em Math. Res. Lett.}, 13(4):571--586, 2006.

\bibitem[PS14]{1210.5254}
J{\'o}zef~H. Przytycki and Radmila Sazdanovi{\'c}.
\newblock Torsion in {K}hovanov homology of semi-adequate links.
\newblock {\em Fund. Math.}, 225:277--304, 2014.

\bibitem[Put13]{1310.1895}
Krzysztof Putyra.
\newblock A 2-category of chronological cobordisms and odd {K}hovanov homology.
\newblock 2013, 1310.1895.

\bibitem[Ras05]{0504045}
Jacob Rasmussen.
\newblock Knot polynomials and knot homologies.
\newblock In {\em Geometry and topology of manifolds}, volume~47 of {\em Fields
  Inst. Commun.}, pages 261--280. Amer. Math. Soc., Providence, RI, 2005.

\bibitem[Ras06]{0607544}
Jacob Rasmussen.
\newblock Some differentials on {K}hovanov-{R}ozansky homology.
\newblock 2006, 0607544.

\bibitem[Ras07]{0508510}
Jacob Rasmussen.
\newblock Khovanov-{R}ozansky homology of two-bridge knots and links.
\newblock {\em Duke Math. J.}, 136(3):551--583, 2007.

\bibitem[Ras10]{0402131}
Jacob Rasmussen.
\newblock Khovanov homology and the slice genus.
\newblock {\em Invent. Math.}, 182(2):419--447, 2010.

\bibitem[Rob13]{1304.0463}
Lawrence Roberts.
\newblock A type {D} structure {K}hovanov homology.
\newblock 2013, 1304.0463.

\bibitem[Rou06]{0409593}
Rapha{\"e}l Rouquier.
\newblock Categorification of {${ {sl}}_2$} and braid groups.
\newblock In {\em Trends in representation theory of algebras and related
  topics}, volume 406 of {\em Contemp. Math.}, pages 137--167. Amer. Math.
  Soc., Providence, RI, 2006.

\bibitem[Rou12]{0812.5023}
Rapha{\"e}l Rouquier.
\newblock Quiver {H}ecke algebras and 2-{L}ie algebras.
\newblock {\em Algebra Colloq.}, 19(2):359--410, 2012.

\bibitem[Sca14]{1401.2093}
Christopher Scaduto.
\newblock Instantons and odd {K}hovanov homology.
\newblock 2014, 1401.2093.

\bibitem[Shu07]{0411643}
Alexander~N. Shumakovitch.
\newblock Rasmussen invariant, slice-{B}ennequin inequality, and sliceness of
  knots.
\newblock {\em J. Knot Theory Ramifications}, 16(10):1403--1412, 2007.

\bibitem[Shu11a]{1101.5614}
Alexander Shumakovtich.
\newblock {K}hovanov homology theories and their applications.
\newblock 2011, 1101.5614.

\bibitem[Shu11b]{1101.5607}
Alexander Shumakovtich.
\newblock Patterns in odd {K}hovanov homology.
\newblock 2011, 1101.5607.

\bibitem[SS06]{0405089}
Paul Seidel and Ivan Smith.
\newblock A link invariant from the symplectic geometry of nilpotent slices.
\newblock {\em Duke Math. J.}, 134(3):453--514, 2006.

\bibitem[Sto07]{0511532}
Marko Sto{\v{s}}i{\'c}.
\newblock Homological thickness and stability of torus knots.
\newblock {\em Algebr. Geom. Topol.}, 7:261--284, 2007.

\bibitem[Str09]{0608234}
Catharina Stroppel.
\newblock Parabolic category {$\mathscr O$}, perverse sheaves on
  {G}rassmannians, {S}pringer fibres and {K}hovanov homology.
\newblock {\em Compos. Math.}, 145(4):954--992, 2009.

\bibitem[Tur06a]{0606464}
Paul Turner.
\newblock Five lectures on {K}hovanov homology.
\newblock 2006, 0606464.

\bibitem[Tur06b]{0411225}
Paul~R. Turner.
\newblock Calculating {B}ar-{N}atan's characteristic two {K}hovanov homology.
\newblock {\em J. Knot Theory Ramifications}, 15(10):1335--1356, 2006.

\bibitem[Vir04]{0202199}
Oleg Viro.
\newblock Khovanov homology, its definitions and ramifications.
\newblock {\em Fund. Math.}, 184:317--342, 2004.

\bibitem[Wat07]{0606630}
Liam Watson.
\newblock Knots with identical {K}hovanov homology.
\newblock {\em Algebr. Geom. Topol.}, 7:1389--1407, 2007.

\bibitem[Wat12]{0807.1341}
Liam Watson.
\newblock Surgery obstructions from {K}hovanov homology.
\newblock {\em Selecta Math. (N.S.)}, 18(2):417--472, 2012.

\bibitem[Wat13a]{1311.1085}
Liam Watson.
\newblock Khovanov homology and the symmetry group of a knot.
\newblock 2013, 1311.1085.

\bibitem[Wat13b]{1010.3051}
Liam Watson.
\newblock New proofs of certain finite filling results via {K}hovanov homology.
\newblock {\em Quantum Topol.}, 4(4):353--376, 2013.

\bibitem[Web07]{0610650}
Ben Webster.
\newblock Khovanov-{R}ozansky homology via a canopolis formalism.
\newblock {\em Algebr. Geom. Topol.}, 7:673--699, 2007.

\bibitem[Web13a]{1309.3796}
Ben Webster.
\newblock Knot invariants and higher representation theory.
\newblock 2013, 1309.3796.

\bibitem[Web13b]{1312.7357}
Ben Webster.
\newblock Tensor product algebras, {G}rassmannians and {K}hovanov homology.
\newblock 2013, 1312.7357.

\bibitem[Weh08]{0409328}
S.~Wehrli.
\newblock A spanning tree model for {K}hovanov homology.
\newblock {\em J. Knot Theory Ramifications}, 17(12):1561--1574, 2008.

\bibitem[Wit12a]{1101.3216}
Edward Witten.
\newblock Fivebranes and knots.
\newblock {\em Quantum Topol.}, 3(1):1--137, 2012.

\bibitem[Wit12b]{1108.3103}
Edward Witten.
\newblock Khovanov homology and gauge theory.
\newblock In {\em Proceedings of the {F}reedman {F}est}, volume~18 of {\em
  Geom. Topol. Monogr.}, pages 291--308. Geom. Topol. Publ., Coventry, 2012.

\bibitem[Wit14]{1401.6996}
Edward Witten.
\newblock Two lectures on the {J}ones polynomial and {K}hovanov homology.
\newblock 2014, 1401.6996.

\bibitem[Wu09]{0612406}
Hao Wu.
\newblock On the quantum filtration of the {K}hovanov-{R}ozansky cohomology.
\newblock {\em Adv. Math.}, 221(1):54--139, 2009.

\bibitem[Wu11]{0907.0695}
Hao Wu.
\newblock Generic deformations of the colored {$ {sl}(N)$}-homology for links.
\newblock {\em Algebr. Geom. Topol.}, 11(4):2037--2106, 2011.

\bibitem[WW09]{0905.0486}
Ben Webster and Geordie Williamson.
\newblock A geometric construction of colored {HOMFLYPT} homology.
\newblock 2009, 0905.0486.

\end{thebibliography}

\end{document}